\documentclass[12pt,reqno]{amsart}

\usepackage{mathrsfs}
\usepackage{amsfonts}
\usepackage{amssymb}
\usepackage{amsthm}
\usepackage{mathtools}
\usepackage{inputenc}

\usepackage[isbn=false, doi=false, url=false, maxbibnames=99]{biblatex}

\mathtoolsset{showonlyrefs,showmanualtags}

\usepackage[inline]{enumitem}
\setlist[enumerate]{itemsep=3pt,topsep=3pt}
\setlist[enumerate,1]{label=(\arabic*)}

\usepackage{multicol}
\usepackage{array, longtable}

\usepackage[spanish, english]{babel}

\usepackage{hyperref}
\hypersetup{colorlinks=true, linkcolor=blue, citecolor = blue, filecolor = blue, urlcolor = blue}

\advance\hoffset by -0.9 truecm

\newcommand{\s}{\vspace{0.2cm}}

\newcommand{\set}[1]{\{#1\}}
\newcommand{\braket}[1]{\langle #1 \rangle}

\numberwithin{equation}{section}

\newtheorem{thm}{Theorem}
\newtheorem{cor}{Corollary}
\newtheorem{lem}{Lemma}
\newtheorem{pro}{Proposition}

\theoremstyle{definition}
\newtheorem{defi}{Definition}

\theoremstyle{remark}
\newtheorem{rem}{Remark}

\newtheorem{que}{Question}

\title{On dihedral group actions on Riemann surfaces}
\author{Pablo Alvarado-Seguel}
\address{Departamento de Matem\'aticas, Facultad de Ciencias, Universidad de Chile, Las Palmeras 3425, Santiago, Chile}
\email{pabloalvarado@ug.uchile.cl}

\author{Sebasti\'an Reyes-Carocca}
\address{Departamento de Matem\'aticas, Facultad de Ciencias, Universidad de Chile, Las Palmeras 3425, Santiago, Chile}
\email{sebastianreyes.c@uchile.cl}

\thanks{This work was partially supported by ANID Fondecyt Regular Grants 1220099 and 1230708, and MATH-AmSud Project 22-MATH-03.}
\date{}

\begin{document}

\begin{abstract}    
    This article deals with dihedral group actions on compact Riemann surfaces and the interplay between different geometric data associated to them.
    First, a bijective correspondence between geometric signatures and analytic representations is obtained.
    Second, a refinement of a result of Bujalance, Cirre, Gamboa and Gromadzki about signature realization is provided.
    Finally, we apply our results to isogeny decompositions of Jacobians by Prym varieties and by elliptic curves, extending results of Carocca, Recillas and Rodr\'iguez.
    In particular, we give a complete classification of Jacobians with dihedral action whose group algebra decomposition induces a decomposition into factors of the same dimension.
\end{abstract}

\maketitle
\thispagestyle{empty}

\section{Introduction} 


The study of  groups of automorphisms of compact Riemann surfaces (or, equivalently, complex projective algebraic curves) and their Jacobian varieties is a classical and extensively considered research area in complex and algebraic geometry.

\s

While the field rests upon foundational results established by Riemann, Klein and Jacobi, among others, a vast majority of the literature has been published within the past three decades. This renewed interest is  based, partially, on  the advent of significant advancements in computer algebra systems and finite group theory, as well as on the variety of points of view that can be considered to study these objects. 

\s
 
 A seminal result in the area claims that each abstract finite group can be realized as a group of automorphisms of some compact Riemann surface, and therefore of some Jacobian variety. Articles aimed at studying group actions of special classes of groups can be found in the literature in plentiful supply. We refer to \cite{@Broughton22}, \cite{@BujalanceCirre21} and \cite{@Rodriguez14} for an up-to-date treatment of this topic.

\s

The geometry of the action of a group on a compact Riemann surface is partially encoded in the so-called {\it signature}, namely, a tuple of nonnegative integers $$(\gamma; m_1, \ldots, m_{v})$$where $\gamma$ is the genus of the corresponding quotient Riemann surface, and $m_1, \ldots, m_v$ are related to the fixed points of the action.

\s

The intimate relationship between a compact Riemann surface $S$ and its Jacobian variety $JS$ is given by the  classical  Torelli theorem. This remarkable result states that $S$ is completely determined by $JS$. More precisely, $$S \cong S' \, \, \iff \,\, JS \cong JS'.$$

We recall that if a group $G$ acts on a compact Riemann surface $S$ then $G$ also acts on the Jacobian variety $JS$. This last action can be represented in the dual of the complex vector space of holomorphic differentials $\Omega^1(S)$ of $S$.
This representation $$\rho_a: G \to \operatorname{GL}(\Omega^1(S)^*)\cong \mbox{GL}(g, \mathbb{C})$$ is called the {\it analytic representation} of $G$.

A group action on $S$ can be used to obtain the {\it group algebra decomposition} of $JS.$ More precisely, a $G$-equivariant isogeny  decomposition $$JS \sim B_1^{n_1}\times \cdots \times B_r^{n_r}$$where each factor $B_j$ is an abelian subvariety of $JS$.  The theory behind this decomposition is a relatively recent development, pioneered by the works of Lange and Recillas \cite{@LangeRecillas04}, and Carocca and Rodríguez \cite{@CaroccaRodriguez06}.
Shortly thereafter, Rojas in \cite{@Rojas07} introduced the notion of {\it geometric signature}, a generalization of the usual signature of an action. The geometric signature  captures more information; for instance,  
the geometric structure of the intermediate covers and the dimension of the subvarieties $B_j$ of the group algebra decomposition.

\s

Dihedral actions on compact Riemann surfaces serve as a rich study case.
Notably, they were foundational examples for the theory of group algebra decompositions of Jacobians, providing valuable insights into the general structure of these decompositions.
Concretely, Recillas and Rodríguez in \cite{@RecillasRodriguez98} worked out the case $ \mathbf{D}_3$, and later, a more general treatment was given by Carocca, Recillas and Rodríguez in \cite{@CaroccaEtAl02}.

\s

In a different line of research, by considering the theory of uniformization,  Fuchsian groups and surface-kernel epimorphisms, Bujalance, Cirre, Gamboa and Gromadzki in \cite{@BujalanceEtAl03} succeeded in providing necessary and sufficient conditions for a signature to admit a dihedral action. 

\s

For related works on dihedral group actions, see \cite{@IzquierdoEtAl19, @LangeEtAl14, @RecillasRodriguez06}.

\s

In this article we deal with dihedral actions on compact Riemann surfaces of genus $g \geqslant 2$.
Our aim is to understand in detail the interplay between different notions of geometric data of an action.
The results of this paper can be summarized as follows. 

\s

{\bf (1)} We prove that there is a bijective correspondence between geometric signatures and analytic representations of dihedral actions on compact Riemann surfaces (Theorem~\ref{thm:bij_geosig_anarep}).
    Explicit formulas are provided.

\s

{\bf (2)} We state necessary and sufficient conditions for a geometric signature to admit a dihedral action (Theorem~\ref{thm-existence_even_0} and Theorem~\ref{thm-existence_even_>0}). 
    These results are a refinement of the results of \cite{@BujalanceEtAl03}.

     \s

{\bf (3)} We solve the problem of deciding when a $\mathbb{C}$-representation is equivalent to the analytic representation of a dihedral action (Theorem~\ref{thm-dihedral_rep_existence_odd} and Theorem~\ref{thm-dihedral_rep_existence_even}).
  
   \s

{\bf (4)} We prove that the dihedral group $\mathbf{D}_n$ is Prym-affordable if and only if $n$ is the power of a prime number (Theorem~\ref{thm:dihedral_Prym_affordable}), extending results of \cite{@CaroccaEtAl02}.
    \s

{\bf (5)}  We characterize the group algebra components of a Jacobian with $\mathbf{D}_n$-action ($n$ odd) that are isogenous to the Prym variety of an intermediate cover (Theorem~\ref{thm:GAD_component_Prym}).

\s

{\bf (6)}  We provide an exhaustive list of geometric signatures of $\mathbf{D}_n$ for which the corresponding group algebra decomposition yields a complete decomposition (namely, in terms of elliptic curves) of the Jacobian (Theorem~\ref{thm:comp_JacDec}).

 \s

{\bf (7)}  We generalize {\bf(6)} and give an exhaustive list of geometric signatures of $\mathbf{D}_n$ for which the corresponding group algebra decomposition provides a decomposition of the Jacobian into factors of the same dimension (Theorem~\ref{thm:kdec_Jac}).

\s
This article is organized as follows.
In Section §\ref{chp:prelim} we shall briefly review the basic background:
group actions on Riemann surfaces, $\mathbb{Q}$-representations, abelian varieties, and representations of automorphism groups.
The  results  will be stated and proved in Section §\ref{chp-dihedral_analytic_rep} (interplay between geometric data) and  Section §\ref{chp:existence} (geometric signature realization).
Finally, Section §\ref{chp:GAD} will be concerned with applications of our results to decomposition of Jacobian varieties. 
\section{Preliminaries}
\label{chp:prelim}

\subsection*{Group actions on Riemann surfaces and Fuchsian groups}

Let $S$ be a compact Riemann surface of genus $g \geqslant 2.$ 
An \emph{action} of a group $G$ on  $S$ is a monomorphism $G \to \operatorname{Aut}(S)$ into the (full) automorphism group of $S$.
A classical result due to Hurwitz states that $G$ is finite and $$|G| \leqslant 84(g-1).$$

Each $G$-action on $S$ induces a \emph{Galois covering} $\pi_G: S \to S_G$, where $S_G$ denotes the quotient Riemann surface  given by the action of $G$ on $S$.
The \emph{signature} of the action is the tuple \begin{equation}\label{signa}(\gamma; m_1, \ldots, m_v),\end{equation} where $\gamma$ is the genus of $S_G$ and $m_1, \ldots, m_v$ are the orders of the stabilizer subgroups of $G$ associated to the $G$-orbits of the ramification points of $\pi_G$.

\s

If the action has signature \eqref{signa} then it satisfies the \emph{Riemann-Hurwitz formula}
\begin{equation}\textstyle
    2g - 2 = |G| [ 2\gamma - 2 + \sum_{j=1}^v ( 1 - \frac{1}{m_j} ) ].
\end{equation}

    Following \cite{@Rojas07}, the \emph{geometric signature} of  $\pi_G: S \to S_G$ is the tuple
    \begin{equation}\label{signageo}
        (\gamma; G_1, \ldots, G_v),
    \end{equation}
    where $\gamma$ is the genus of $S_G$, and $G_1, \ldots, G_v$ are the conjugacy classes of the stabilizer subgroups of $G$ associated to the $G$-orbits of the ramification points of $\pi_G$.

\s 

For any pair of subgroups $H \leqslant K$ of $G$, the induced maps $$\pi_K^H: S_H \to S_K$$are called \emph{intermediate coverings} of $\pi_G$.
The genus and the ramification data of the intermediate coverings are determined by the geometric signature of the $G$-action.
We refer to \cite[§3]{@Rojas07} for more details.

\s 

Let $\mathbb{H}$ denote the upper half-plane.
    A \emph{Fuchsian group} $\Delta$ is a discrete subgroup of $$\operatorname{Aut}(\mathbb{H})\cong \operatorname{\mathbb{P}SL}(2, \mathbb{R}).$$
    A \emph{surface Fuchsian group} is a torsion free Fuchsian group.

\s 

Let $\Delta$ be a co-compact Fuchsian group, that is, $\mathbb{H}_{\Delta}$ is compact.
The universal covering map $\mathbb{H} \to \mathbb{H}_{\Delta}$ is unramified if and only if $\Delta$ is torsion free.
If $\gamma$ denotes the genus of $\mathbb{H}_{\Delta}$ and $m_1, \ldots, m_v$ are the ramification indices of the branch values of the covering map $\Delta \to \mathbb{H}_\Delta$, then the tuple $$s(\Delta):=(\gamma; m_1, \ldots, m_v)$$ is the \emph{signature} of $\Delta$. In addition, in such a case, $\Delta$ has a canonical presentation
\begin{equation}\label{eq-fuchs_canonical_rep}\textstyle
    \langle \alpha_1, \beta_1 \ldots, \alpha_\gamma, \beta_\gamma, x_1, \ldots, x_v \mid x_1^{m_1} = \cdots = x_v^{m_v} = \prod_{i=1}^\gamma [\alpha_i, \beta_i] \prod_{j=1}^v x_j = 1 \rangle,
\end{equation}
where the  brackets denotes the commutator.
The elements $\alpha_1, \beta_1, \ldots, \alpha_\gamma,\beta_\gamma$ are  the \emph{hyperbolic generators} of $\Delta$, whereas $x_1, \ldots, x_v$ are the \emph{elliptic generators} of $\Delta$.

\s

Let $S$ be a compact Riemann surface of genus $g \geqslant2$. The uniformization theorem states that there exists a co-compact surface Fuchsian group $\Gamma$ such that $$S \cong  \mathbb{H}_\Gamma.$$

In addition, Riemann's existence theorem ensures that a group $G$ acts on $S \cong \mathbb{H}_\Gamma$ with signature $\sigma$ if and only if there exists a Fuchsian group $\Delta$ of signature $\sigma$ and an exact sequence of groups
        \begin{equation}
            1 \to \Gamma \to \Delta \xrightarrow{\theta} G \to 1.
        \end{equation}
        In this case $S_G \cong \mathbb{H}_\Delta$ and $g$ satisfies the Riemann-Hurwitz formula.

Note that the signature of $\Delta$ and the signature of the action of $G$ agree.
We say that the action is represented by the \emph{surface-kernel epimorphism} (hereafter, ske for short) $\theta: \Delta \to G$.
It is common to identify $\theta$ with the tuple or \emph{generating vector}
\begin{equation}
    (\theta(\alpha_1), \theta(\beta_1), \ldots, \theta(\alpha_\gamma), \theta(\beta_\gamma), \theta(x_1), \ldots, \theta(x_v)) \in G^{2\gamma+v}.
\end{equation}

We refer to \cite{@Broughton22}, \cite{@FarkasKra92} and \cite{@Miranda95} for more details concerning compact Riemann surfaces and group actions.

\subsection*{Irreducible $\mathbb{Q}$-representations}

Let $G$ be a finite group and let $\mathbb{F}$ be a field of characteristic zero.
An \emph{$\mathbb{F}$-representation} of $G$ is a group homomorphism $\rho: G \to \operatorname{GL}(V)$ into the general linear group of a finite-dimensional  $\mathbb{F}$-vector space $V$.
We will usually abuse notation and simply write $V$ instead of $\rho$.
The \emph{degree} $d_V$ of $V$ is the dimension of $V$ as an $\mathbb{F}$-vector space, and the \textit{character} $\chi_V$ of $V$ is the map$$\chi_V : G \to \mathbb{C} \mbox{ such that } g \mapsto \mbox{trace of }V(g).$$ 
Two representations are \textit{equivalent} if their characters agree; we write $V_1 \cong V_2$. The \textit{character field} $K_V$ of $V$ is the extension of $\mathbb{F}$ by the values of the character of $V$.
The \emph{Schur index} $s_V$ is the smallest positive integer such that there exists a degree $s_V$ extension of fields $L_V>K_V$ over which $V$ can be defined.

\s

We denote by $\operatorname{Irr}_\mathbb{F}(G)$ the set formed by the irreducible $\mathbb{F}$-representations of $G$.
If $V$ is an $\mathbb{F}$-representation and $\operatorname{Irr}_\mathbb{F}(G) = \set{U_1, \ldots, U_v}$, then for each $1 \leqslant j \leqslant v$ there exists a unique nonnegative integer $a_j$, the \emph{multiplicity} of $U_j$ in $V$, such that
\begin{equation}
    V = a_1 U_1 \oplus \cdots \oplus a_v U_v,
\end{equation}
where $a_j U_j = U_j \oplus \overset{a_j}{\cdots} \oplus U_j$.
The integer $a_j$ agrees with $\braket{V, U_j}/\braket{U_j, U_j}$, where
\begin{equation}\textstyle
    \braket{V, U} = \tfrac{1}{|G|} \sum_{g \in G} \chi_V(g) \overline{\chi_U}(g).
\end{equation}

It is known that for $W \in \operatorname{Irr}_\mathbb{Q}(G)$ there exists $V \in \operatorname{Irr}_\mathbb{C}(G)$ such that
\begin{equation}
    W \otimes_\mathbb{Q} \mathbb{C} \cong s_V (\oplus_\sigma V^\sigma) = (\oplus_\sigma V^\sigma) \oplus \overset{s_V}{\cdots} \oplus (\oplus_\sigma V^\sigma),
\end{equation}
where the direct sum is taken over the Galois group associated to the field extension $K_V > \mathbb{Q}$.
We say that $V$ and $W$ are \emph{Galois associated}.

\s

Let $H$ be a subgroup of $G$.
The \emph{fixed subspace} of $V$ under $H$ is
\begin{equation}
    V^H = \set{ v\in V \mid \rho(h)(v)=v\ \text{for all}\ h \in H};
\end{equation}
we denote its dimension by $d_V^H$.
We refer to \cite{@Serre77} and \cite{@Steinberg12} for more details.

\subsection*{Abelian varieties}

A $g$-dimensional \textit{complex torus} $X = V_\Lambda$ is the quotient of a $g$-dimensional $\mathbb{C}$-vector space $V$ by a discrete subgroup $\Lambda$ of $V$ of maximal rank.
Each complex torus is an abelian group and a $g$-dimensional compact connected complex analytic manifold.
A {\it homomorphism} between complex tori is a holomorphic map which is  also a group homomorphism.
If $f: X_1 \to X_2$ is a  tori homomorphism then it is induced by a unique $\mathbb{C}$-linear map $$\rho_a(f): V_1 \to V_2 \mbox{ such that }\rho_a(f)(\Lambda_1) \subset \Lambda_2 .$$
The restriction of $\rho_a(f)$ to $\Lambda_1$ is a $\mathbb{Z}$-linear map $\rho_r(f): \Lambda_1 \to \Lambda_2$. The maps $$
\rho_a : \operatorname{Hom}(X_1,X_2) \to \operatorname{Hom}_\mathbb{C}(V_1, V_2) \mbox{ and }
\rho_r : \operatorname{Hom}(X_1, X_2) \to \operatorname{Hom}_\mathbb{Z}(\Lambda_1, \Lambda_2)$$are called the \textit{analytic representation} and \textit{rational representation} of $\operatorname{Hom}(X_1, X_2)$. 
Both representations can be extended to
\begin{equation}
    \operatorname{Hom}_\mathbb{Q}(X_1, X_2) := \operatorname{Hom}(X_1, X_2) \otimes_\mathbb{Z} \mathbb{Q}.
\end{equation}
The (extension to $\mathbb{C}$ of the) rational representation is equivalent to the direct sum of the analytic representation and its complex conjugate:
\begin{equation}\label{eq-rational-analytic-decomposition}
    \rho_r \otimes 1 \cong \rho_a \oplus \overline{\rho_a}.
\end{equation}

An \textit{isogeny} of tori is a surjective homomorphism with finite kernel.
Two isogenous tori are denoted by $X_1 \sim X_2$.

\s

An \textit{abelian variety} is a complex torus which is also a complex projective algebraic variety.
The \textit{Jacobian variety} $JS$ of a compact Riemann surface $S$ of genus $g$ is an (irreducible principally polarized) abelian variety of dimension $g$: it is the quotient
\begin{equation}
    JS = \Omega^1(S)^*/\Lambda
\end{equation}
of linear functionals on the space of holomorphic differentials modulo periods.

As mentioned in the introduction, Torelli’s theorem states that
\begin{equation}
    S \cong S' \iff JS \cong JS',
\end{equation}
where $JS \cong JS'$ is an isomorphism of principally polarized abelian varieties.

\s

Given a non-constant holomorphic map $f: S_1 \to S_2$ between compact Riemann surfaces, the pullback $f^*: JS_2 \to JS_1$ is an isogeny onto its image $f^*(JS_2)$.
By Poincare's theorem there exists an abelian subvariety $P(f)$ of $f^*(JS_2)$ such that
\begin{equation}
    JS_1 \sim JS_2 \times P(f).
\end{equation}The subvariety $P(f)$ is called the \emph{Prym variety} of $f$. We refer to \cite{@BirkenhakeLange04} for more details on abelian varieties.

\subsection*{Representations of groups of automorphisms}

Let $G$ be a finite group, and assume that $G$ acts on a compact Riemann surface $S$.
Then there is an induced action of $G$ on $JS$.
Without loss of generality, the composite maps
$$G \to \operatorname{Aut}(JS) \xrightarrow{\rho_a} \operatorname{GL}(\Omega^1(S)^*)  \mbox{ and }G \to \operatorname{Aut}(JS) \xrightarrow{\rho_r} \operatorname{GL}(\Lambda \otimes_\mathbb{Z} \mathbb{Q})$$
are called the \emph{analytic representation} and the \emph{rational representation} of the action.
Abusing notation, we denote them by $\rho_a$ and $\rho_r$.

\s

Set $\zeta_n = e^{2\pi i/n}$ and consider the function
$\mathcal{N}: \operatorname{Irr}_\mathbb{C}(G) \times G \to \mathbb{Q}$
defined as
$$\textstyle \mathcal{N}(V,g) = \sum_{\alpha=1}^{|g|} N_{g,\alpha} \frac{|g|-\alpha}{|g|}$$
where $N_{g,\alpha}$ is the number of eigenvalues of $V(g)$ that are equal to $\zeta_{|g|}^\alpha$. Assume that the action of $G$ on $S$ has signature \eqref{signa}  and is represented by the ske $\theta: \Delta \to G$.
If $V \in \operatorname{Irr}_\mathbb{C}(G)$ is nontrivial, then the  Chevalley-Weil formula \cite{@ChevalleyWeil34} states that
\begin{equation}\textstyle
    \braket{\rho_a,V} = d_V(\gamma-1) + \sum_{j=1}^v \mathcal{N}(V,c_j),
\end{equation}
where $c_1, \ldots, c_v$ are the images of the $v$ elliptic generators of $\Delta$.
If $V$ is the trivial representation then $\braket{\rho_a, V} = \gamma$.

\s
    
Similarly, following \cite[Theorem 5.10]{@Rojas07}, if the geometric signature of the action is  \eqref{signageo} and $V \in \operatorname{Irr}_\mathbb{C}(G)$ is nontrivial, then
\begin{equation}\textstyle
    \braket{\rho_r, V} = 2 d_V (\gamma-1) + \sum_{k=1}^v (d_V - d_V^{G_k}).
\end{equation}
If $V$ is the trivial representation then $\braket{\rho_r, V} = 2 \gamma$.

\s

\paragraph{\bf Notation.} If $g \in G$ and $a \in \mathbb{Z}_+$, then the symbol $\braket{g}^a$ in a geometric signature abbreviates $\braket{g}, \overset{a}{\ldots}, \braket{g}$.
Similarly, if $m \geqslant2$ is an integer, then the symbol $m^a$ in a signature abbreviates $m, \overset{a}{\ldots}, m$.
\section{Geometric signature and analytic representation}\label{chp-dihedral_analytic_rep}

This section is devoted to study  the interplay between the analytic representation and the geometric signature of dihedral actions.

\subsection*{Preliminaries}
\label{sec:int_div}

For each $n \in \mathbb{Z}_{+}$ we write
\begin{equation}
    \mathbb{Z}^{|n} := \set{q \in \mathbb{Z}_+ : q\ \text{divides}\ n}.
\end{equation}
    
We say that $q \in \mathbb{Z}_+$ is a \emph{$k$-divisor} of $n$ if $q \in \mathbb{Z}^{|n} $  and $n/q$ is a product of exactly $k$ distinct prime numbers. We denote the set of all $k$-divisors of $n$ by
\begin{equation}
    \mathbb{Z}_k^{|n} = \set{q \in \mathbb{Z}^{|n}: q\ \text{is a}\ k\text{-divisor of}\ n}.
\end{equation}
By definition $\mathbb{Z}_0^{|n} = \set{n}$.

\begin{defi}\label{def-div_transform}
    Let $\Psi, \Phi: \mathbb{Z}_+ \to \mathbb{Z}$ be two functions.
    \begin{enumerate}
        \item The \emph{divisor transform} of $\Psi$ is the function $\widehat{\Psi}: \mathbb{Z}_+ \to \mathbb{Z}$ given by
        \begin{equation}\textstyle
            \widehat{\Psi}(n) = \sum_{q \in \mathbb{Z}^{|n}} \Psi(q).
        \end{equation}
        \item The \emph{inverse divisor transform} of $\Phi$ is the function $\widetilde{\Phi}: \mathbb{Z}_+ \to \mathbb{Z}$ given by
        \begin{equation}\textstyle
            \widetilde{\Phi}(n) = \sum_{k \geqslant 0} (-1)^k \sum_{q \in \mathbb{Z}_k^{|n}} \Phi(q).
        \end{equation} 
    \end{enumerate}
\end{defi}

\begin{rem} If the prime decomposition of $n$ is $p_1^{\alpha_1} \cdots p_r^{\alpha_r}$, then for each $1 \leqslant k \leqslant r$,
    \begin{equation}\textstyle
        \sum_{q \in \mathbb{Z}_k^{|n}} \Phi(q) = \sum_{1 \leqslant j_1 < \cdots < j_k \leqslant r} \Phi ( \frac{n}{p_{j_1} \cdots p_{j_k}} ).
    \end{equation}
    In particular, $\sum_{q \in \mathbb{Z}_0^{|n}} \Phi(q) = \Phi(n)$ and $\sum_{q \in \mathbb{Z}_r^{|n}} \Phi(q) = \Phi(\frac{n}{p_1\cdots p_r})$.
\end{rem}

\begin{lem}\label{lem-div_transform}
    Let $\Psi, \Phi:\mathbb{Z}^+\to\mathbb{Z}$ be two functions and let  $n \in \mathbb{Z}_+$. Then
     $$\widehat{\Phi}(q)=\Psi(q) \mbox{ for all }  q \in \mathbb{Z}^{|n} \,\, \iff \,\,  \widetilde{\Psi}(q)=\Phi(q) \mbox{ for all } q \in \mathbb{Z}^{|n}.$$
   
\end{lem}

\begin{proof}
    $(\Rightarrow)$
    Let $q \in \mathbb{Z}^{|n}$ with prime decomposition $q = p_1^{\alpha_1} \cdots p_r^{\alpha_r}$ ($\alpha_j \geqslant 1$ for $1 \leqslant j \leqslant r$).
    For $k \geqslant 0$, the hypothesis implies that
    \begin{equation}\label{eq:sum_kdiv}\textstyle
        \sum_{s \in \mathbb{Z}_k^{|q}} \Psi(s) =
        \sum_{s \in \mathbb{Z}_k^{|q}} \widehat{\Phi}(s) =
        \sum_{s \in \mathbb{Z}_k^{|q}} \sum_{t' \in \mathbb{Z}^{|s}} \Phi(t').
    \end{equation}
    Note that if $t$ is not a divisor of $q$ then $\Phi(t)$ does not appear in \eqref{eq:sum_kdiv}.
    
    For $t \in \mathbb{Z}^{|q}$, let us write $q/t = p_1^{\beta_1} \cdots p_r^{\beta_r}$, where $0 \leqslant \beta_j \leqslant \alpha_j$ for $1 \leqslant j \leqslant r$.
    Observe that if $q/t$ has $r' = \# \set{\beta_j \geqslant 1}$ distinct prime factors then the number of times that $\Phi(t)$ appears in \eqref{eq:sum_kdiv} is the binomial coefficient $\binom{r'}{k}$.
    It follows that if $t$ is different from $q$ then $\Phi(t)$ appears
    \begin{equation}\textstyle
        \sum_{k=0}^{r'} (-1)^k \binom{r'}{k} = (1-1)^{r'} = 0
    \end{equation}
    times in the total sum $\widetilde{\Psi}(q) = \sum_{k=0}^r(-1)^k \sum_{s \in \mathbb{Z}_k^{|q}} \Psi(s)$.
    The term $\Phi(q)$ appears exactly once, which proves the desired result.

    \s

    $(\Leftarrow)$
    The proof is by strong induction on the number $N$ of prime factors (counting multiplicities) of $q \in \mathbb{Z}^{|n}$.
    The cases $N = 0,1$ hold trivially.
    Let us assume that we have proved the result for all divisors of $n$ with less than $N \geqslant 2$ prime factors, and choose $q \in \mathbb{Z}^{|n}$ with $N$ prime factors.
    By the assumption and the strong inductive hypothesis, we have that
    \begin{align}
        \Phi(q) = \widetilde{\Psi}(q)
        &= \textstyle \Psi(q) + \sum_{k=1}^r (-1)^k \sum_{s \in \mathbb{Z}_k^{|q}} \Psi(s) \\
        &= \textstyle \Psi(q) + \sum_{k=1}^r (-1)^k \sum_{s \in \mathbb{Z}_k^{|q}} \widehat{\Phi}(s) \\
        &= \textstyle \Psi(q) + \sum_{k=1}^r (-1)^k \sum_{s \in \mathbb{Z}_k^{|q}} \sum_{t' \in \mathbb{Z}^{|s}} \Phi(t'),\label{eq:sum_kdiv2}
    \end{align}
    where $r$ is the number of distinct prime factors of $q$.
    Using the same counting argument of the previous part of the proof, we conclude that for each proper divisor $t$ of $q$, $\Phi(t)$ appears
    \begin{equation}\textstyle
        \sum_{k=1}^r (-1)^k \binom{r}{k} = (1-1)^r - 1 = -1
    \end{equation}
    times in the sum $\eqref{eq:sum_kdiv2}$.
    It follows that $\widehat{\Phi}(q) = \Psi(q)$, as desired.
\end{proof}

\begin{pro}\label{pro-div_transform}
    $\Phi\mapsto\widehat{\Phi}$ and $\Phi\mapsto\widetilde{\Phi}$ are inverse operations.
\end{pro}

\begin{proof} It follows directly from the previous lemma.
\end{proof}

\subsection*{Irreducible $\mathbb{C}$-representation of the dihedral group} Let $n \geqslant 2$ be an integer, and consider the dihedral group  $$\mathbf{D}_n = \braket{r,s \mid r^n= s^2= (sr)^2=1}$$of order $2n$.
The irreducible $\mathbb{C}$-representations of $\mathbf{D}_n$ are well-known. If $n$ is even then $\mathbf{D}_n$ has four $\mathbb{C}$-representations of degree one, given by$$\psi_1 : r \mapsto 1,\ s \mapsto 1; \,\,  
    \psi_2 : r \mapsto 1,\ s \mapsto -1; \,\, 
    \psi_3 : r \mapsto -1,\ s \mapsto 1; \,\,
    \psi_4 : r \mapsto -1,\ s \mapsto -1, $$and $(n-2)/2$ irreducible $\mathbb{C}$-representations of degree two, given by
\begin{equation}
    \rho^h: r \mapsto \operatorname{diag}(\omega^h, \overline{\omega}^h),\ s \mapsto \begin{psmallmatrix} 0 & 1 \\ 1 & 0 \end{psmallmatrix},
\end{equation}
where $\omega = \zeta_n = e^{2\pi i/n}$ and $1 \leqslant h \leqslant (n-2)/2$.

\s

If $n$ is odd then $\mathbf{D}_n$ has two $\mathbb{C}$-representations of degree one and $(n-1)/2$ irreducible $\mathbb{C}$-representations of degree two, given by $$\psi_1, \psi_2 \mbox{ and } \rho^h \mbox{ for } 1 \leqslant h \leqslant (n-1)/2.$$

\begin{rem}
    As the character field of each irreducible $\mathbb{C}$-representation of $\mathbf{D}_n$ is real, one has  $\rho_a \cong \overline{\rho_a}$ for each dihedral action.
    Moreover, $\rho_a$ is determined by the geometric signature of the action and
    \begin{equation}
        \braket{\rho_r, V} = 2 \braket{\rho_a, V} \text{ for each } V \in \operatorname{Irr}_\mathbb{C}(G).
    \end{equation}
\end{rem}

The following auxiliary lemma provides the value of the function $$\mathcal{N}: \operatorname{Irr}_\mathbb{C}(\mathbf{D}_n) \times \mathbf{D}_n \to \mathbb{Q}$$ introduced in \S\ref{chp:prelim}. Note that $\mathcal{N}(\psi_1, g) = 0$ for each $g \in \mathbf{D}_n$.

\begin{lem}\label{lem-CW_odd}
    Let $n\geqslant 2$ be an  integer and let $q \in \mathbb{Z}^{|n} - \set{1}$.
    Then the value of the function $\mathcal{N}: \operatorname{Irr}_\mathbb{C}(\mathbf{D}_n) \times \mathbf{D}_n \to \mathbb{Q}$ is summarized in the following tables.
\begin{enumerate}
    \item If $n$ is odd then $$\begin{array}{l|ccc}
        \mathcal{N} & s & r^{n/q} \\
        \hline
        \psi_1 & 0 & 0 \\
        \psi_2 & 1/2 & 0 \\
        \rho^h & 1/2 & \varepsilon
    \end{array}$$
    where $\varepsilon=0$ if $q$ divides $h$, and $\varepsilon=1$ otherwise.

    \item If $n$ is even then 
    \begin{equation}
    \begin{array}{l|cccc}
        \mathcal{N} & s & sr & r^{n/q} \\
        \hline
        \psi_1 & 0 & 0 & 0 \\
        \psi_2 & 1/2 & 1/2 & 0 \\
        \psi_3 & 0 & 1/2 & \delta/2 \\
        \psi_4 & 1/2 & 0 & \delta/2 \\
        \rho^h & 1/2 & 1/2 & \varepsilon
    \end{array}    
    \end{equation}
    where $\delta = 0$ if $2q$ divides $n$, and $\delta = 1$ otherwise; and $\varepsilon = 0$ if $q$ divides $h$ and $\varepsilon = 1$ otherwise.
    In particular, for $r^{n/2}$ we have that $\delta = 0$ if $n \in 4\mathbb{Z}$, and $\delta = 1$ otherwise; and $\varepsilon = 0$ if $h$ is even, and $\varepsilon = 1$ otherwise.
    \end{enumerate}
\end{lem}

\begin{proof}
It is a routine computation and follows directly from the definition of $\mathcal{N}.$ 
\end{proof}

\subsection*{Analytic representation formulas}\label{sec-dihedral_analytic}

Let $S$ be a compact Riemann surface of genus $g \geqslant 2$ endowed with a dihedral action represented by the ske $\theta: \Delta \to \mathbf{D}_n$.

The geometric signature of the action has the form
\begin{equation}
    (\gamma; \braket{s}^a, \braket{sr}^b, C_1, \ldots, C_v) \mbox{ for some integers } a,b \geqslant 0,
\end{equation}where  $C_j = \braket{r^{n/m_j}}$ is a cyclic subgroup of $\mathbf{D}_n$ of order $m_j \geqslant 2$.

\s

Observe that if $n$ is odd then the signature and the geometric signature encode the same information. In such a case $\braket{s} \sim \braket{sr}$ and the relevant number is $t := a+b$.

\begin{defi}
    
    The \emph{signature function} $\Psi_\theta: \mathbb{Z}_+ \to \mathbb{Z}$ of the action $\theta$ is given by
    \begin{equation}
        \Psi_\theta(q) = \# \set{ 1 \leqslant j \leqslant v: m_j=q}.
    \end{equation}
\end{defi}

Note that $\widehat{\Psi}_\theta(n) = v$ and $\widehat{\Psi}_\theta(1) = 0$.

\begin{lem}\label{lem-div_transform_signature_func}
    Let $\theta: \Delta \to \mathbf{D}_n$ be a ske, and let $q$ be a positive integer.
    The following statements hold.
    \begin{enumerate}
        \item $\widehat{\Psi}_\theta(q) = \widehat{\Psi}_\theta((n,q))$.
        \item $\widehat{\Psi}_\theta(n) = \widehat{\Psi}_\theta(\tfrac{n}{q})$ if and only if $\operatorname{lcm}(m_1, \ldots, m_v)$ divides $\frac{n}{q}$, for $q \in \mathbb{Z}^{|n}$.
        \item $\widehat{\Psi}_\theta(n) - \widehat{\Psi}_\theta(q)$ is the number of cyclic subgroups $C_j$ appearing in the geometric signature of the action such that $m_j$ does not divide $q$.
    \end{enumerate}
\end{lem}

\begin{proof}
    The statement (3) follows directly from the definition of $\widehat{\Psi}_\theta$.
    If $t \in \mathbb{Z}_+$ does not divide $n$, then $\Psi_\theta(t) = 0$.
    It follows that $\widehat{\Psi}_\theta(q)$ is equal to the sum of $\Psi_\theta$ over the common divisors of $q$ and $n$.
    That is,
    \begin{equation}\textstyle
        \widehat{\Psi}_\theta(q) =
        \sum_{t\in\mathbb{Z}^{|q}} \Psi_\theta(t) =
        \sum_{t \in \mathbb{Z}^{|q} \cap \mathbb{Z}^{|n}} \Psi_\theta(t) =
        \sum_{t \in \mathbb{Z}^{|(n,q)}} \Psi_\theta(t) =
        \widehat{\Psi}_\theta((n,q)),
    \end{equation}
    proving (1).
    To prove (2), observe that $\widehat{\Psi}_\theta(n) = \widehat{\Psi}_\theta(\frac{n}{q})$ if and only if $\set{m_1, \ldots, m_v} \subset \mathbb{Z}^{|n/q}$.
    This in turn holds if and only if $\operatorname{lcm}(m_1, \ldots, m_v)$ divides $n/q$.
\end{proof}

Now, we determine the analytic representation explicitly in terms of the geometric signature of the action.

\begin{thm}\label{thm-dihedral_odd_formula}
    Let $n\geqslant 3$ be an odd integer, and let $S$ be a compact Riemann surface of genus $g \geqslant 2$ with dihedral action represented by the surface-kernel epimorphism $\theta: \Delta \to \mathbf{D}_n$.
    If the action has geometric signature $(\gamma; 2^t, m_1, \ldots, m_v)$, then its analytic representation is given by
    \begin{equation}
        \rho_a \cong \gamma \psi_1 \oplus \mu_2 \psi_2 \oplus \bigoplus_{h=1}^{(n-1)/2} \nu_h \rho^h,    \end{equation}where 
$\mu_2 = \gamma - 1 + \tfrac{1}{2}t$ and $$\nu_h = 2(\gamma-1) + \tfrac{1}{2}t + \widehat{\Psi}_\theta(n) - \widehat{\Psi}_\theta(h) \mbox{ for each }1 \leqslant h \leqslant \tfrac{n-1}{2}.$$
\end{thm}

\begin{proof}
    Observe that the sum $$\sum_{q \in \mathbb{Z}^{|n}} \Psi_\theta(q) \mathcal{N}(\rho^h, r^{n/q})$$ is the number of cyclic groups $C_j$ appearing in the geometric signature such that $m_j$ does not divide $h$.
    Thus, the proof of the theorem follows from Lemma~\ref{lem-CW_odd}, Lemma~\ref{lem-div_transform_signature_func} and the Chavelley-Weil formula.
\end{proof}

Observe that if $p \geqslant 3$ is prime and  $\mathbf{D}_p$ acts with signature $(\gamma; 2^t, p^l)$, then
\begin{equation}\label{eq:analytic_rep_prime_formula}
    \rho_a \cong \gamma \psi_1 \oplus (\gamma-1+\tfrac{1}{2}t) \psi_2 \oplus [2(\gamma-1)+\tfrac{1}{2}t + l] (\rho^1 \oplus \cdots \oplus \rho^{(p-1)/2})
\end{equation}

\begin{thm}\label{thm-dihedral_even_formula}
    Let $n\geqslant 2$ be an even integer, and let $S$ be a compact Riemann surface of genus $g \geqslant 2$ with a dihedral action represented by a surface-kernel epimorphism $\theta: \Delta \to \mathbf{D}_n$.
    If the action has geometric signature $(\gamma; \braket{s}^a, \braket{sr}^b, C_1, \ldots, C_v)$, then its analytic representation is given by
    \begin{equation}
        \rho_a \cong \gamma \psi_1 \oplus \mu_2 \psi_2 \oplus \mu_3 \psi_3 \oplus \mu_4 \psi_4 \oplus \bigoplus_{h=1}^{(n-2)/2} \nu_h \rho^h,
    \end{equation}
    where
    \begin{align}
        \mu_2 &=  \gamma - 1 + \tfrac{1}{2} (a+b), \\
        \mu_3 &=  \gamma - 1 + \tfrac{1}{2} [b+\widehat{\Psi}_\theta(n) - \widehat{\Psi}_\theta
        (\tfrac{n}{2})], \\
        \mu_4 &=  \gamma - 1 + \tfrac{1}{2} [a + \widehat{\Psi}_\theta(n) - \widehat{\Psi}_\theta(\tfrac{n}{2})], \\
        \nu_h &=  2(\gamma-1) + \tfrac{1}{2} (a+b) + \widehat{\Psi}_\theta(n) - \widehat{\Psi}_\theta(h),
    \end{align}
    for each $1 \leqslant h \leqslant (n-2)/2$.
\end{thm}

\begin{proof} After noticing that $\widehat{\Psi}_\theta(n) - \widehat{\Psi}_\theta(\tfrac{n}{2})$ is the number of cyclic subgroups $C_j$ appearing in the geometric signature generated by an odd power of $r$, the proof follows analogously as  the proof of Theorem \ref{thm-dihedral_odd_formula}.
\end{proof}

\subsection*{Remark: Topological actions and signatures} 

Let $S_j$ be a Riemann surface with an action $\varepsilon_j: G \to \operatorname{Aut}(S_j)$, for $j=1,2$.
The actions $\varepsilon_1$ and $\varepsilon_2$ are \emph{topologically equivalent} if there exist a group automorphism $\omega \in \operatorname{Aut}(G)$ and an orientation-preserving homeomorphism $T: S_1 \to S_2$ such that
\begin{equation}
    T \varepsilon_1(\omega(g)) = \varepsilon_2(g) T
\end{equation}
for all $g \in G$.
If $T$ is holomorphic then we speak of \emph{analytic equivalence}.

\s

Roughly speaking, the compact Riemann surfaces of a fixed genus that admit an action of the same group with the same topological class form an irreducible subvariety  of the moduli space of Riemann surfaces; see \cite{@Gonzalez-DiezHarvey92}.
The topological classification of  actions was the key ingredient  to provide a stratification of the moduli space of compact Riemann surfaces; see \cite{@Broughton90a}.

\begin{pro}\label{pro:topeq_geosig}
    If two actions are topologically equivalent, then their geometric signatures either agree or differ by an outer automorphism of $G$.
\end{pro}

\begin{proof}
    As the geometric signature is preserved under inner group automorphisms, we only need to verify that conjugation by orientation-preserving homeomorphisms does not affect the stabilizer groups. Let $T: S_1 \to S_2$ be an orientation-preserving homeomorphism, and let $\epsilon_j$ denote a $G$-action on $S_j$  such that $\epsilon_2(g) = T \circ \epsilon_1(g) \circ T^{-1}$ for all $g \in G$.
    We observe that, for $p \in S_1$, the automorphism $\epsilon_1(g)$ fixes $p$ if and only if $\epsilon_2(g) = T \circ \epsilon_1(g) \circ T^{-1}$ fixes $T(p)$, as desired.
\end{proof}

Since the analytic representation is determined by the geometric signature, we see that if $G$ does not have outer automorphisms, then the proposition above says that the analytic representation is constant over classes of topological equivalence of $G$-actions. In particular, we obtain the following corollary.

\begin{cor}

 Let $n \geqslant 3$ be an odd integer.
   
  \begin{enumerate} 
        \item Analytic representations of  $\mathbf{D}_n$-actions are constant over classes of topological equivalence.
        \item $\mathbf{D}_n$-actions that are topologically (hence, analytically) distinct with the same signature share the same analytic representation.
    \end{enumerate}
\end{cor}

\subsection*{Geometric signature formulas} We have determined the analytic representation of a dihedral action in terms of its  geometric signature. Now, we deal with the converse problem.

\begin{defi}\label{def-pre_signature_func} Let $n \geqslant 2$ be an integer and
    let $V$ be a $\mathbb{C}$-representation of $\mathbf{D}_n$.
    The \emph{pre-signature function} $\Phi_V: \mathbb{Z}_+ \to \mathbb{Z}$ of $V$ is given by
    \begin{equation}
        \Phi_V(q) =
        \begin{cases}
            \braket{V, \rho^1} - \braket{V,  \psi_1 \oplus \psi_2} + 1 & \text{if}\ (n,q) = n \\
            \braket{V, \rho^1} - \braket{V, \psi_3 \oplus \psi_4} & \text{if}\ n\ \text{is even}\ \text{and}\ (n,q) = \tfrac{n}{2} \\
            \braket{V, \rho^1} - \braket{V, \rho^{(n,q)}} & \text{if}\ (n,q) < \tfrac{n}{2}
        \end{cases}
    \end{equation}
\end{defi}

The pre-signature function $\Phi_V$ is dual to  the signature function $\Psi_\theta$ in the following sense.

\begin{pro}\label{pro:presignature_transform}
    If $\rho_a$ is the analytic representation of a dihedral action represented by a ske $\theta: \Delta \to \mathbf{D}_n$, then 
    \begin{equation}
        \Psi_\theta = \widetilde{\Phi}_{\rho_a}
        \mbox{ and } \, \widehat{\Psi}_\theta = \Phi_{\rho_a}.
    \end{equation}
\end{pro}

\begin{proof}
    Assume that $n$ is even.
    Observe that as a consequence of Theorem~\ref{thm-dihedral_even_formula} 
    \begin{equation}
        \begin{aligned}
            \widehat{\Psi}_\theta(n) &= \braket{\rho_a, \rho^1} -\braket{\rho_a, \psi_1 \oplus \psi_2} + 1, \\
            \widehat{\Psi}_\theta(\tfrac{n}{2}) &= \braket{\rho_a, \rho^1} - \braket{\rho_a, \psi_3 \oplus \psi_4}, \\
            \widehat{\Psi}_\theta(h) &= \braket{\rho_a, \rho^1} - \braket{\rho_a, \rho^h},
        \end{aligned}
    \end{equation}
    for $1 \leqslant h \leqslant (n-2)/2$.
    Besides, by Lemma~\ref{lem-div_transform_signature_func}, we see that
    \begin{equation}
        \widehat{\Psi}_\theta(q) = \widehat{\Psi}_\theta((n,q)) = \Phi_{\rho_a}(q) \quad \text{for}\ q \in \mathbb{Z}_+.
    \end{equation}
    We conclude by Proposition~\ref{pro-div_transform}.
    The proof for the case $n$ odd is analogous.
\end{proof}

An immediate consequence is that we obtain an explicit formula for the geometric signature of the action in terms of its analytic representation.

\begin{pro}\label{pro-dihedral_inverse_formula_odd}
    Let $n \geqslant 3$ be an odd integer, and let $S$ be a Riemann surface of genus $g \geqslant 2$ with a dihedral action represented by a ske $\theta: \Delta \to \mathbf{D}_n$.
   Assume that $$\mathbb{Z}^{|n}-\{1\} = \set{m_1, \ldots, m_v} \mbox{ and that } 1 < m_1 < \ldots < m_v.$$
    If $\rho_a$ is the analytic representation of the action, then the signature of the action
    \begin{equation}
        (\gamma; 2^t, m_1^{l_1}, \ldots, m_v^{l_v})
    \end{equation}
    is given by
    \begin{equation}
     \gamma = \braket{\rho_a, \psi_1}, \,\, 
            t = 2 \braket{\rho_a, \psi_2} - 2 \braket{\rho_a, \psi_1} + 2 \,\mbox{ and } \,
            l_j = \widetilde{\Phi}_{\rho_a}(m_j)\ \text{for}\ 1 \leqslant j \leqslant v.
    \end{equation}
\end{pro}

\begin{proof}
    The formulas for $\gamma$ and $t$ are a direct consequence of Theorem~\ref{thm-dihedral_odd_formula}.
    Besides, by definition, $l_j = \Psi_\theta(m_j)$ for $1 \leqslant j \leqslant v$. It follows that $$l_j = \Psi_\theta(m_j) \implies  \widetilde{\Phi}_{\rho_a}(l_j) = \widetilde{\Phi}_{\rho_a}(\Psi_\theta(m_j))=m_j,$$where the last equality holds by Proposition~\ref{pro:presignature_transform}.
\end{proof}

\begin{pro}\label{pro-dihedral_inverse_formula_even}
    Let $n \geqslant 2$ be an even integer, and let $S$ be a Riemann surface of genus $g \geqslant 2$ with a dihedral action represented by a ske $\theta: \Delta \to \mathbf{D}_n$.
     Assume that $$\mathbb{Z}^{|n} -\{1\}= \set{m_1, \ldots, m_v} \mbox{ and that } 1 < m_1 < \ldots < m_v.$$
    If $\rho_a$ is the analytic representation of the action, then the geometric signature of the action
    \begin{equation}
        (\gamma; \braket{s}^a, \braket{sr}^b, \braket{r^{n/m_1}}^{l_1}, \ldots, \braket{r^{n/m_v}}^{l_v})
    \end{equation}
    is given by
    \begin{equation}
            \gamma = \braket{\rho_a, \psi_1}, \,\, 
            a = \braket{\rho_a, \psi_2 \oplus \psi_4} - \braket{\rho_a, \psi_1 \oplus \psi_3} + 1, \,\,
            b = \braket{\rho_a, \psi_2 \oplus \psi_3} - \braket{\rho_a, \psi_1 \oplus \psi_4} + 1,    \end{equation} 
       and $l_j = \widetilde{\Phi}_{\rho_a}(m_j)$ for $1 \leqslant j \leqslant v.$
\end{pro}

\begin{proof}
 The proof is analogous to the one of the      previous proposition.
\end{proof}

After verifying that the analytic representation formulas (Theorems \ref{thm-dihedral_odd_formula} and \ref{thm-dihedral_even_formula}) and the geometric signature formulas (Propositions \ref{pro-dihedral_inverse_formula_odd} and \ref{pro-dihedral_inverse_formula_even}) are inverse formulas, we have the proof of the following result.

\begin{thm}\label{thm:bij_geosig_anarep}
    There is a bijective correspondence between geometric signatures and analytic representations of dihedral actions on  Riemann surfaces of genus $g \geqslant 2$.
\end{thm}

\section{Existence of dihedral actions}
\label{chp:existence}

Bujalance, Cirre, Gamboa and Gromadzki in \cite{@BujalanceEtAl03} studied dihedral actions and provided necessary and  sufficient conditions for a signature to admit a surface-kernel epimorphism onto a dihedral group. It is worth emphasizing that their results provide sufficient, but not necessary, conditions under which there exist actions with a given geometric signature.

\s

In this section we provide a refinement of their results from signatures to geometric signatures.
We then apply our results to address the problem of deciding when a $\mathbb{C}$-representation is the analytic representation of a dihedral action.

\subsection*{Geometric signature realization} Since for dihedral groups of order $2n$ with $n$ odd, signature and geometric signature are equivalent, hereafter in this section we assume $n$ to be even.

\s

Let $\Delta$ be a Fuchsian group of signature
$$(\gamma; 2^{a+b}, m_1, \ldots, m_v) \mbox{ with } m_j \in \mathbb{Z}^{|n} - \set{1}$$
canonically presented by hyperbolic generators $\alpha_1, \beta_1, \ldots, \alpha_\gamma, \beta_\gamma$, elliptic generators $x_1, \ldots, x_{a+b}$, $y_1, \ldots, y_v$, and relations
\begin{equation}\label{eq:fuchsian_def_rel}\textstyle
    x_i^2 = y_j^{m_j} = \prod_{t=1}^\gamma [\alpha_t, \beta_t] \prod_{k=1}^{a+b} x_k \prod_{l=1}^v  y_l = 1,
\end{equation}
for $i = 1, \ldots, a+b$ and $j = 1, \ldots, v$.
Assume that there exists a ske $\theta: \Delta \to \mathbf{D}_n$ that represents an action with geometric signature
\begin{equation}
    (\gamma; \braket{s}^a, \braket{sr}^b, C_1, \ldots, C_v),
\end{equation}
where $C_j = \braket{r^{n/m_j}}$ is a cyclic subgroup of order $m_j \geqslant 2$. Without loss of generality, we can assume  that
\begin{equation}
    \theta(x_k) =
    \begin{cases}
        sr^\text{even}, & 1 \leqslant k \leqslant a, \\
        sr^\text{odd}, & a+1 \leqslant k \leqslant a+b,
    \end{cases}
\end{equation}
and $\theta(y_j) = (r^{n/m_j})^{q_j}$ with $(q_j, m_j) = 1$ for $1 \leqslant j \leqslant v$.
We define
 $$A:= \# \set{1 \leqslant j \leqslant v : n/m_j\ \text{is  odd}} \, \mbox{ and } \, B := \# \set{1 \leqslant j \leqslant v : n/2m_j\ \text{is  odd}}.$$

\begin{lem}\label{lem:dihedral_epi}
    With the previous notations.
    \begin{enumerate}
        \item There are integers $\xi_1$, $\xi_2$ and $\xi_3$ such that
        \begin{equation}\label{eq:dihedral_epi_rel1}\textstyle
            \prod_{t=1}^\gamma \theta([\alpha_t, \beta_t]) = r^{2\xi_1},\
            \prod_{j=1}^{a+b} \theta(x_j) = r^{\xi_2},\ 
            \prod_{j=1}^{v} \theta(y_j) = r^{\xi_3},
        \end{equation}
        and $2\xi_1 + \xi_2 + \xi_3 \in n \mathbb{Z}$.
        Concretely, $\xi_3 = \sum_{j=1}^v nq_j/m_j$.
        \item The integers $a$, $b$, $\xi_2$, $\xi_3$ and $A$ have the same parity;
        $a+b$ is even.
        \item If $A=0$ then every integer  $nq_j/m_j$ is even.
        In this case, if $n \in 4\mathbb{Z}$ then the integers $\xi_3/2$ and $B$ have the same parity.
    \end{enumerate}
\end{lem}

\begin{proof}
    {(1)} We have that $\prod_{j=1}^v \theta(y_j) = r^{\xi_3}$ with $\xi_3 = \sum_{j=1}^v nq_j/m_j$.
    Since the derived subgroup of $\mathbf{D}_n$ is $\braket{r^2}$, it follows that $\prod_{t=1}^\gamma \theta([\alpha_t, \beta_t]) = r^{2\xi_1}$ for some $\xi_1 \in \mathbb{Z}$.
    Thus, the  relation \eqref{eq:fuchsian_def_rel}  implies that $\prod_{j=1}^{a+b} \theta(x_j) = r^{\xi_2}$ for $\xi_2 \in \mathbb{Z}$ and $2\xi_1 + \xi_2 + \xi_3 \in n \mathbb{Z}$.
    
    \s
    
    {(2)} As $\prod_{j=1}^{a+b} \theta(x_j) = r^{\xi_2}$, $a$ and $b$ must have the same parity.
    The parity of $\xi_2$ is determined by the number of terms of the form $sr^\text{odd}$ in the product $\prod_{j=1}^{a+b} \theta(x_j)$.
    Concretely, $\xi_2$ is odd if and only if $b$ is odd.
    By the relation above $\xi_2$ and $\xi_3$ must have the same parity.
    Now, observe that $\xi_3 = \sum_{j=1}^v nq_j/m_j$ is odd if and only if there is an odd number of integers $1 \leqslant j \leqslant v$ such that $nq_j/m_j$ is odd.
    Since $(q_j, m_j) = 1$, $nq_j/m_j$ is odd if and only if $n/m_j$ is odd.
    We conclude that $\xi_3$ and $A$ have the same parity.
    
    \s
    
    {(3)} Assume that $A = 0$.
    Thus, every integer  $n/m_j$ is even and the same holds for $nq_j/m_j$.
    The integer $\xi_3/2 = \sum_{j=1}^v nq_j/2m_j$ is odd if and only if there is an odd number of integers $1 \leqslant j \leqslant v$ such that $nq_j/2m_j$ is odd.
    Now, assume that $n \in 4 \mathbb{Z}$.
    To conclude it suffices to show that $nq_j/2m_j$ is odd if and only if $n/2m_j$ is odd.
    Indeed, if $n/2m_j$ is odd then $m_j$ is even, hence $q_j$ is odd and so is $nq_j/2m_j$.
    The converse is direct.
\end{proof}

\begin{thm}\label{thm-existence_even_0}
    Let $n \geqslant 2$ be an even integer and let $\gamma = 0$.
    Then necessary and sufficient conditions for the existence of a surface-kernel epimorphism $\theta: \Delta \to \mathbf{D}_n$ of geometric signature $(0; \braket{s}^a, \braket{sr}^b, C_1, \ldots, C_v)$ are
    \begin{enumerate}
        \item $a$, $b$ and $A$ have the same parity,
        \item $a+b \geqslant 2$ is even,
        \item if $a+b=2$ then $\operatorname{lcm}(m_1, \ldots, m_v)=n$, and
        \item if $a+b>2$, and $a=0$ or $b=0$ then $A>0$.
    \end{enumerate}
\end{thm}

\begin{proof}
    We start by proving that the conditions are necessary.
    \begin{enumerate}
        \item This has already been proved in Lemma~\ref{lem:dihedral_epi}.
        \item Assume that $a+b<2$.
        Then $a=b=0$ and $\theta(\Delta) \leqslant \braket{r}$, a contradiction with the surjectivity of $\theta$.
        \item Assume that $a+b=2$.
        By Lemma~\ref{lem:dihedral_epi} we have that $\braket{\theta(x_1), \theta(x_2)} = \braket{sr^k, r^{\xi_3}}$ for some $k \in \mathbb{Z}$ and $\xi_3 = \sum_{j=1}^v nq_j/m_j$.
        Observe that $$\theta(\Delta) = \braket{sr^k, r^{n/\operatorname{lcm}(m_1, \ldots, m_v}}.$$
        Since $\theta$ is onto it follows that $\operatorname{lcm}(m_1, \ldots, m_v) = n$.
        \item Assume that $a+b>2$, and $a=0$ or $b=0$.
        Since $a$ and $b$ have the same parity, either $a \geqslant 4$ and $b=0$, or $b \geqslant 4$ and $a=0$.
        If $a \geqslant 4$ and $b=0$ then $$\braket{\theta(x_1), \ldots, \theta(x_a)} \leqslant \braket{s, r^2}.$$
        Since $\theta$ is onto, there must exist some $\theta(y_j) = r^{nq_j/m_j}$ with $n/m_j$ odd, and therefore $A \neq 0$. 
        The same argument holds for $a=0$ and $b \geqslant 4$.
    \end{enumerate}

    We now show that the conditions are sufficient.
    Let $a$ and $b$ be two nonnegative integers as in Condition (2).
    Then, one of the following cases occurs:
    \begin{multicols}{2}
    \begin{enumerate}[label=(\roman*)]
        \item $a, b \geqslant 2$ are even,
        \item $a \geqslant 3$ and $b \geqslant 1$ are odd,
        \item $a \geqslant 1$ and $b \geqslant 3$ are odd,
        \item $a \geqslant 4$ is even and $b = 0$,
        \item $a = 0$ and $b \geqslant 4$ is even,
        \item $a = 2$ and $b = 0$,
        \item $a = 0$ and $b = 2$,
        \item $a = b = 1$.
    \end{enumerate}
    \end{multicols}
    Set $\xi_3 = \sum_{j=1}^v n/m_j$.
    The following tuples are generating vectors for each one of the indicated cases:
    \begin{enumerate}
        \item[(i)] $(s, \overset{a}{\ldots}, s, sr, \overset{b-1}{\ldots}, sr, sr^{1-\xi_3}, r^{n/m_1}, \ldots, r^{n/m_v})$,
        \item[(ii)] $(s, \overset{a-1}{\ldots}, s, sr^{1+\xi_3}, sr, \overset{b}{\ldots}, sr, r^{n/m_1}, \ldots, r^{n/m_v})$,
        \item[(iv)] $(s, \overset{a-2}{\ldots}, s, sr^2, sr^{2-\xi_3}, r^{n/m_1}, \ldots, r^{n/m_v})$,
        \item[(vi,viii)] $(s, sr^{-\xi_3}, r^{n/m_1}, \ldots, r^{n/m_v})$.
    \end{enumerate}
    The remaining cases follow after considering the outer automorphism of $\mathbf{D}_n$ given by $r\mapsto r$, $s \mapsto sr$.
\end{proof}

\begin{thm}\label{thm-existence_even_>0}
    Let $n \geqslant 2$ be an even integer and let $\gamma>0$.
    Then necessary and sufficient conditions for the existence of a surface-kernel epimorphism $\theta: \Delta \to \mathbf{D}_n$ of geometric signature $(\gamma; \braket{s}^a, \braket{sr}^b, C_1, \ldots, C_v)$ are
    \begin{enumerate}
        \item $a$, $b$, and $A$ have the same parity ($a+b$ is even), and
        \item if $\gamma=1$ and $a=b=0$ then $\operatorname{lcm}(m_1, \ldots, m_v)=n$ or $n/2$. In the latter case, if $n \in 4\mathbb{Z}$ then $B$ is odd.
    \end{enumerate}
\end{thm}

\begin{proof}
    Let us prove that the conditions are necessary.
    Condition (1) has already been proved in Lemma~\ref{lem:dihedral_epi}.
    In order to show Condition (2), assume that $\gamma=1$ and $a = b = 0$.
    As $\theta$ is onto at least one of the hyperbolic generators of $\Delta$ must be sent to $\mathbf{D}_n \setminus \braket{r}$.
    After considering an outer automorphism of $\mathbf{D}_n$ of the form $r \mapsto r$, $s \mapsto sr^i$, we can assume that $\theta(\alpha_1) = s$.
    Note that $\theta(\beta_1) = sr^{\xi_1}$ or $r^{-\xi_1}$ for some $\xi_1 \in \mathbb{Z}$.
    In each case, $\theta([\alpha_1, \beta_1]) = r^{2\xi_1}$.
    By Lemma~\ref{lem:dihedral_epi}, the integers $\xi_1$ and $\xi_3 = \sum_{j=1}^v nq_j/m_j$ satisfy the relation
    \begin{equation}\label{eq:xi1_xi3_rel}
        2\xi_1 + \xi_3 \in n \mathbb{Z}.
    \end{equation}
    Thus, $$r^{2\xi_1} \in \braket{\theta(y_1), \ldots, \theta(y_v)} = \braket{r^{n/\operatorname{lcm}(m_1, \ldots, m_v}}.$$
    Surjectivity requires that $\braket{s, r^{\xi_1}, r^{n/\operatorname{lcm}(m_1, \ldots, m_v)}} = \mathbf{D}_n$, and hence $\braket{s, r^{n/\operatorname{lcm}(m_1, \ldots, m_v)}}$ has index 1 or 2 in $\mathbf{D}_n$. Equivalently,
    \begin{equation}
        \operatorname{lcm}(m_1, \ldots, m_v) = n\ \text{or}\ n/2.
    \end{equation}
    
    Assume that $\operatorname{lcm}(m_1, \ldots, m_v) = n/2$.
    It follows that each $n/m_j$ is even ($A=0$) and $\xi_1$ is odd 
    (otherwise $\theta(\Delta) = \braket{s, r^{\xi_1}, r^2} \neq \mathbf{D}_n$).
    Now, assume that $n \in 4\mathbb{Z}$.
    Relation \eqref{eq:xi1_xi3_rel} turns into $ \xi_1 + \xi_3/2 \in \tfrac{n}{2} \mathbb{Z}.$ Since $\xi_1$ is odd and $n/2$ is even, $\xi_3/2$ must be odd.
    We conclude by Lemma~\ref{lem:dihedral_epi}.

    \s

    Now, we show that the conditions are sufficient.
    Let $a$ and $b$ be nonnegative integers as in Condition (1).
    Then, one of the following cases occurs:
    \begin{multicols}{2}
    \begin{enumerate}[label=(\roman*)]
        \item $\gamma \geqslant 2$;
        \item $\gamma = 1$, $a \neq 0$;
        \item $\gamma = 1$, $b \neq 0$;
        \item $\gamma = 1$, $a = b = 0$.
    \end{enumerate}
    \end{multicols}
    
    
    Set $\xi_2 = 0$ if $b$ is even and $\xi_2 = 1$ if $b$ is odd.
    We also set $\xi_3 = \sum_{j=1}^v n/m_j$.
    The following tuples are generating vectors for the indicated cases:
    \begin{enumerate}
        \item[(i)] $(s, r^{(\xi_2+\xi_3)/2}, r, \overset{2\gamma-2}{\ldots}, r; s, \overset{a}{\ldots}, s, sr, \overset{b}{\ldots}, sr, r^{n/m_1}, \ldots, r^{n/m_v})$,
        \item[(ii)] $(sr, r^{(\xi_2+\xi_3)/2}; s, \overset{a}{\ldots}, s, sr, \overset{b}{\ldots}, sr, r^{n/m_1}, \ldots, r^{n/m_v})$,
        \item[(iv)] $(s, r^{\xi_1}; r^{n/m_1}, \ldots, r^{n/m_v})$ with $\xi_1 = (\delta n + \xi_3)/2$ for $\delta \in \set{0,1}$.
    \end{enumerate}
    (We can always choose $\delta$ so that the tuple above is a generating vector.)
    Case (iii) follows after considering the outer automorphism $r \mapsto r$, $s \mapsto sr$.
\end{proof}

\subsection*{Analytic representation criteria}

The following results answer the question of when a given $\mathbb{C}$-representation is the analytic representation of a dihedral action. In order to state them, we consider the support $$\operatorname{Supp} \Psi = \set{ q \in \mathbb{Z}_+: \Psi(q) \neq 0}$$  of the function $\Psi: \mathbb{Z}_+ \to \mathbb{Z}$.

\begin{thm}\label{thm-dihedral_rep_existence_odd}
    Let $n \geqslant 3$ be odd and let $V$ be a  $\mathbb{C}$-representation  of $\mathbf{D}_n.$ Then $V$ is equivalent to the analytic representation of a $\mathbf{D}_n$-action if and only if the following statements hold.
    \begin{enumerate}
        \item $\braket{V, \psi_2} + 1 \geqslant \braket{V, \psi_1}$.
        \item $\widetilde{\Phi}_V(q) \geqslant 0$ for each $q \in \mathbb{Z}^{|n} - \set{1}$.
        \item $\braket{V, \rho^h} = \braket{V, \rho^{(n,h)}}$ for $1 \leqslant h \leqslant (n-1)/2$.
        \item if $\braket{V, \psi_1} \leqslant 1$ and $\braket{V, \psi_2} = 0$ then $\operatorname{lcm}(\operatorname{Supp} \widetilde{\Phi}_V) = n$.
    \end{enumerate}
   Furthermore, if we write
   $$\mathbb{Z}^{|n}-\{1\} = \set{m_1, \ldots, m_v} \mbox{ with } 1 < m_1 < \ldots < m_v,$$
   then the corresponding action has signature
    \begin{equation}
        (\braket{V, \psi_1}; 2^t, m_1^{l_1}, \ldots, m_v^{l_v}),
    \end{equation}
    where $t = 2(\braket{V, \psi_2} - \braket{V, \psi_1} + 1)$ and $l_j = \widetilde{\Phi}_V(m_j)$ for $j = 1, \ldots, v$.
\end{thm}

\begin{proof}
    Let us prove that the conditions are necessary.
    Assume that $V$ is the analytic representation of the action represented by the ske $\theta: \Delta \to \mathbf{D}_n$ of signature $$(\gamma; 2^t, m_1^{l_1}, \ldots, m_v^{l_v}),$$ where $l_j \geqslant 0$.
    Conditions (1) and (2) are clear from Proposition~\ref{pro-dihedral_inverse_formula_odd}.
    Since $\widehat{\Psi}_\theta(q) = \widehat{\Psi}_\theta((n,q))$ for $q \in \mathbb{Z}_+$, Theorem~\ref{thm-dihedral_odd_formula} implies Condition (3).
    After applying Theorem~\ref{thm-dihedral_odd_formula}, Condition (4) is a direct consequence of \cite[Theorems~2.1 and 2.3]{@BujalanceEtAl03}.

    \s
    Let us now prove that the conditions are sufficient.
    Let $V$ be a $\mathbb{C}$-representation of $\mathbf{D}_n$ that satisfies Conditions $(1), \ldots, (4)$.
    As $V$ satisfies (1) and (2), the formulas given in Proposition~\ref{pro-dihedral_inverse_formula_odd} induce a well-defined signature.
    Using Condition (4) and \cite[Theorems~2.1 and 2.3]{@BujalanceEtAl03}, it can be proved that the signature realizes as a dihedral action.
    Finally, we need to check that $V$ is equivalent to the analytic representation $\rho_a$ of the action.
    By Theorem~\ref{thm-dihedral_odd_formula} and Proposition~\ref{pro-dihedral_inverse_formula_odd}, it follows that $$\braket{\rho_a, \psi_j} = \braket{V, \psi_j} \mbox{ for } j=1,2 \mbox{ and } \widetilde{\Phi}_{\rho_a}(q) = \widetilde{\Phi}_V(q)\mbox{ for }q \in \mathbb{Z}^{|n}.$$
    Then, Lemma~\ref{lem-div_transform} and Proposition~\ref{pro-div_transform} imply that $\Phi_{\rho_a}(q) = \Phi_V(q)$ for $q \in \mathbb{Z}^{|n}$, and thus $\braket{\rho_a, \rho^{(n,h)}} = \braket{V, \rho^{(n,h)}}$ for $1 \leqslant h \leqslant (n-1)/2$.
    We conclude by Condition (3).
\end{proof}

\begin{lem}\label{lem:sign_func_AB}
    Let $n \geqslant 2$ be an even integer.
    \begin{enumerate}
        \item $A = \widehat{\Psi}_\theta(n) - \widehat{\Psi}_\theta(\tfrac{n}{2}) = \Phi_{\rho_a}(n) - \Phi_{\rho_a}(\tfrac{n}{2})$.
        \item $B = \widehat{\Psi}_\theta(\tfrac{n}{2}) - \widehat{\Psi}_\theta(\tfrac{n}{4}) = \Phi_{\rho_a}(\tfrac{n}{2}) - \Phi_{\rho_a}(\tfrac{n}{4})$ for $n \in 4 \mathbb{Z}$.
    \end{enumerate}
\end{lem}

\begin{proof}
    It is a consequence of Lemma~\ref{lem-div_transform_signature_func} together with the fact that $\widehat{\Psi}_\theta = \Phi_{\rho_a}$.
\end{proof}

\begin{thm}\label{thm-dihedral_rep_existence_even}
Let $n \geqslant 2$ be even and let $V$ be a  $\mathbb{C}$-representation  of $\mathbf{D}_n.$ Then $V$ is equivalent to the analytic representation of a $\mathbf{D}_n$-action if and only if the following statements hold.

    \begin{enumerate}
        \item $\braket{V, \psi_2} + 1 \geqslant \braket{V, \psi_1} + | \braket{V, \psi_3} - \braket{V, \psi_4} |$.
        \item $\widetilde{\Phi}_V(q) \geqslant 0$ for each $q \in \mathbb{Z}^{|n} - \set{1}$.
        \item $\braket{V, \rho^h} = \braket{V, \rho^{(n,h)}}$ for $1 \leqslant h \leqslant (n-2)/2$.
        \item if $\braket{V, \psi_1} = \braket{V, \psi_2} = 0$ then $\operatorname{lcm}(\operatorname{Supp} \widetilde{\Phi}_V) = n$.
        \item if $\braket{V, \psi_1} = 0$, $\braket{V, \psi_2} \geqslant 1$ and $|\braket{V, \psi_3} - \braket{V, \psi_4}| = \braket{V, \psi_2} + 1$ then $\Phi_V(n) > \Phi_V(\tfrac{n}{2})$.
        \item if $\braket{V, \psi_1} = 1$ and $\braket{V, \psi_2} = 0$ then $\operatorname{lcm}(\operatorname{Supp} \widetilde{\Phi}_V) = n$ or $n/2$.
        In the latter case, if $n \in 4 \mathbb{Z}$ then $\Phi_V(\tfrac{n}{2}) - \Phi_V(\tfrac{n}{4})$ is odd.
    \end{enumerate}
    Furthermore, if we write$$\mathbb{Z}^{|n} -\{1\}= \set{m_1, \ldots, m_v} \mbox{ with } 1 < m_1 < \ldots < m_v,$$
    then the corresponding action has geometric signature
    \begin{equation}
        (\braket{V, \psi_1}; \braket{s}^a, \braket{sr}^b, \braket{r^{n/m_1}}^{l_1}, \ldots, \braket{r^{n/m_v}}^{l_v}),
    \end{equation}
    where $l_j = \widetilde{\Phi}_V(m_j)$ for $j = 1, \ldots, v$, and
    $$
        a = \braket{V, \psi_2 \oplus \psi_4} - \braket{V, \psi_1 \oplus \psi_3} + 1, \,\,
        b = \braket{V, \psi_2 \oplus \psi_3} - \braket{V, \psi_1 \oplus \psi_4} + 1.
    $$
\end{thm}

\begin{proof}
    The proof is analogous to the one of the previous theorem. We recall that we must use the results associated to $n$ even: Theorem~\ref{thm-dihedral_even_formula} instead of Theorem~\ref{thm-dihedral_odd_formula}, Proposition~\ref{pro-dihedral_inverse_formula_even} instead of Proposition~\ref{pro-dihedral_inverse_formula_odd}, and Theorems~\ref{thm-existence_even_0} and \ref{thm-existence_even_>0} instead of \cite[Theorems~2.2 and 2.3]{@BujalanceEtAl03}.
    Also, Lemma~\ref{lem:sign_func_AB} gives equivalent statements about the geometric signature in terms of the pre-signature function.
\end{proof}

We end this section by considering the question of whether or not an irreducible $\mathbb{C}$-representation arises as the analytic representation of a $\mathbf{D}_n$-action.
\begin{lem}\label{lem-rho1}
    Let $n \geqslant 3$ be an integer. If $\rho_a$ is the analytic representation of a $\mathbf{D}_n$-action in genus $g \geqslant 2$, then $\braket{\rho_a, \rho^1} \geqslant 1$.
\end{lem}

\begin{proof}
    Assume that $n \geqslant 4$ is even.
    By Theorem~\ref{thm-dihedral_even_formula} we know that
    \begin{equation}
        \braket{\rho_a, \rho^1} = 2(\gamma-1) + \tfrac{1}{2}(a+b) + \widehat{\Psi}_\theta(n).
    \end{equation}
    
    If $\braket{\rho_a, \rho^1} = 0$ then either
    \begin{enumerate}
        \item $\gamma=1$, $a+b=0$ and $\widehat{\Psi}_\theta(n)=0$,
        \item $\gamma=0$, $a+b=2$ and $\widehat{\Psi}_\theta(n)=1$, or
        \item $\gamma=0$, $a+b=4$ and $\widehat{\Psi}_\theta(n)=0$.
    \end{enumerate}
    
    Note that Theorem~\ref{thm-existence_even_>0} tells us that no action satisfies (1).
    Besides, by Theorem~\ref{thm-existence_even_0}, the only geometric signature compatible with (2) is $(0; \braket{s}, \braket{sr}, \braket{r})$.
    Also, (3) has geometric signature $(0; \braket{s}^2, \braket{sr}^2)$.
    In each case the genus $g$ is less than 2, a contradiction.
    If $n\geqslant 3$ is odd then we set $t := a+b$ and obtain the same conclusions.
\end{proof}

\begin{pro}
    There is a $\mathbf{D}_n$-action in genus $g \geqslant 2$ whose analytic representation $\rho_a$ is irreducible if and only if $n \in \set{3,4,6}$.
    In each case, $\rho_a \cong \rho^1$ and the action is in genus $2$.
\end{pro}

\begin{proof}
    Assume that $\rho_a$ is irreducible.
    By Lemma~\ref{lem-rho1} one has that $\braket{\rho_a, \rho^1}$ is always positive, hence $\rho_a \cong \rho^1$.
    However, by Theorem~\ref{thm-dihedral_rep_existence_odd} (if $n$ is odd) and Theorem~\ref{thm-dihedral_rep_existence_even} (if $n$ is even) one has that  $$\braket{\rho_a, \rho^1} \neq 0 \implies \braket{\rho_a, (\rho^1)^\sigma} \neq 0$$for all $\sigma$ in the Galois group of $\rho^1$.
    It follows that the Galois group of $\rho^1$ must be trivial, and this only happens for $\mathbf{D}_n$ with $n \in \set{3, 4, 6}$.
    Conversely, consider the signature $\sigma_3 = (0; 2, 2, 3, 3)$ for $\mathbf{D}_3$ and the geometric signatures
    \begin{equation}
        \sigma_4 = (0; \braket{s}, \braket{sr}, \braket{r^2}, \braket{r})\ \text{for}\ \mathbf{D}_4\ \text{and}\
        \sigma_6 = (0; \braket{s}, \braket{sr}, \braket{r^3}, \braket{r^2})\ \text{for}\ \mathbf{D}_6.
    \end{equation}
    In each case, the geometric signature realizes as an action and its associated analytic representation satisfies $\rho_a \cong \rho^1$.
\end{proof}
\section{Group algebra decomposition}
\label{chp:GAD}

Carocca, Recillas and Rodr\'iguez in \cite{@CaroccaEtAl02} studied  Riemann surfaces with dihedral actions and provided the associated group algebra decomposition of their Jacobians.
In this section, equipped with tools not available at that time, we deal with the problem of determining when such a decomposition is affordable by Prym varieties.

\s

Along the way, we recover some of the results of \cite{@CaroccaEtAl02}. We also relate our results with the classical Ekedahl-Serre problem of completely decomposable Jacobians.

\subsection*{Preliminaries}
\label{sec:group_algebra_decomposition} Let  $G$ be a finite group and let $A$ be an abelian variety with $G$-action. Following \cite[Theorem 2.2]{@LangeRecillas04a}, there are abelian subvarieties $B_1, \ldots, B_r$ of $A$ and a $G$-equivariant isogeny
    \begin{equation}\label{chao}
        A \sim B_1^{n_1} \times \cdots \times B_r^{n_r}
    \end{equation}where $r$ is the number of irreducible $\mathbb{Q}$-representation. 
    Furthermore, if $$\operatorname{Irr}_\mathbb{Q}(G) = \set{W_1, \ldots, W_r}$$then $G$ acts on $B_j^{n_j}$ via the representation $W_j$. The isogeny \eqref{chao} is called the  \emph{group algebra decomposition} of $A$ respect to $G$, and the subvarieties $B_j$, which are defined up to isogeny, are called the \emph{group algebra components} of $A$ associated to $W_j$.

\s

If $A=JS$ and $W_1$ is  the trivial representation (as we shall do in the sequel) then $$B_1 \sim JS_G \mbox{ and }n_1=1.$$In addition, as proved in \cite{@Rojas07}, if $V_j \in \operatorname{Irr}_\mathbb{C}(G)$ is Galois associated to $W_j$ then 
    \begin{equation}\label{mesa}
        \dim B_j = \tfrac{1}{2} k_{V_j} \braket{\rho_r, V_j} \mbox{ for } j=2, \ldots, r,
    \end{equation}
    where $k_{V_j} = s_{V_j} |\operatorname{Gal}(K_{V_j}/\mathbb{Q})|$.

\s

The group algebra decomposition of $JS$ with respect to $G$ induces isogeny decompositions of the Jacobian variety of the quotients of $S$ as well as of the Prym varieties of the intermediate coverings. More precisely, following the results of \cite{@CaroccaRodriguez06}, if  $H \leqslant K$ are subgroups of $G$ then
    \begin{equation}\label{indu}
        JS_H \sim JS_G \times B_2^{u_2} \times \cdots B_r^{u_r},
    \end{equation}
    where $u_j = d_{V_j}^{H}/ s_{V_j}$ for $2 \leqslant j \leqslant r$.
    In addition,
    \begin{equation}
        P(\pi_K^H) \sim B_2^{t_2} \times \cdots \times B_r^{t_r},
    \end{equation}
    where $t_j = (d_{V_j}^H - d_{V_j}^K)/ s_{V_j}$ for $2 \leqslant j \leqslant r$. In particular, $n_j=d_{V_j}/s_{v_j}$ in \eqref{chao}.

\s

We refer to \cite{@LangeRodriguez22} for more details on decomposition of Jacobians by Prym varieties.

\subsection*{Jacobians with dihedral actions}

We recall that the Schur index of the dihedral representations is equal to 1.
Also, there is a bijection between $\mathbb{Z}^{|n} - \set{1,2}$ and the set of irreducible $\mathbb{Q}$-representations of $\mathbf{D}_n$ of degree greater than one, given by
\begin{equation}
    \mathbb{Z}^{|n} - \set{1,2} \to \operatorname{Irr}_\mathbb{Q}(\mathbf{D}_n), \quad
    q \mapsto W(q) := \oplus_\sigma(\rho^{n/q})^\sigma.
\end{equation}

The general case for dihedral groups is as follows.
For $n$ even, the group algebra decomposition of $JS$ with respect to $\mathbf{D}_n$ is
\begin{equation}\textstyle
    JS \sim JS_{\mathbf{D}_n} \times B_2 \times B_3 \times B_4 \times \prod_{q \in \mathbb{Z}^{|n} - \set{1,2}} B(q)^2,
\end{equation}
where $B_j$ is the group algebra component associated to the nontrivial degree one representation $\psi_j$, and $B(q)$ is the group algebra component associated to $W(q)$.
If $n$ is odd then we just omit $B_3$ and $B_4$, and hence
\begin{equation}\textstyle
    JS \sim JS_{\mathbf{D}_n} \times B_2 \times \prod_{q \in \mathbb{Z}^{|n} - \set{1,2}} B(q)^2.
\end{equation}

\s 

Hereafter, $\phi$ denotes the \emph{Euler totient function}.

\begin{pro}\label{pro-dihedral_dim_B(q)}
    The dimension of the group algebra components are
     $$\dim B_j = \braket{\rho_a, \psi_j}\, \mbox{ and } \, \dim B(q) = \tfrac{1}{2} \phi(q) \braket{\rho_a, \rho^{n/q}}$$ for  $j=1, \ldots, 4$ and 
        $q \in \mathbb{Z}^{|n} - \set{1,2}$.
\end{pro}

\begin{proof} The proof follows from  \eqref{mesa}  coupled with the fact that 
     $\rho_a \cong \overline{\rho_a}$. 
\end{proof}

The dimensions $d_V^H$ are known for the dihedral groups.
We include them here for latter use.

\begin{lem}\label{lem-dihedral_fix_dim}
    Let $H$ be a subgroup of $\mathbf{D}_n$ and let  $V \in \operatorname{Irr}_\mathbb{C}(\mathbf{D}_n)$.
    For $\alpha,q \in \mathbb{Z}^{|n}$
    \begin{equation}
        \begin{array}{c|ccccc}
            d_V^H & \psi_1 & \psi_2 & \psi_3 & \psi_4 & \rho^{n/q} \\
            \hline
            \braket{r^{n/\alpha}} & 1 & 1 & \delta & \delta & 2\varepsilon \\
            \braket{s, r^{n/\alpha}} & 1 & 0 & \delta & 0 & \varepsilon \\
            \braket{sr, r^{n/\alpha}} & 1 & 0 & 0 & \delta & \varepsilon
        \end{array}
    \end{equation}
    where $\delta=1$ if $n/\alpha$ is even, and $\delta = 0$ otherwise;
    and $\varepsilon=1$ if $\alpha$ divides $n/q$, and $\varepsilon = 0$ otherwise.
    If $n$ is odd then we  omit columns $\psi_3$ and $\psi_4$.
\end{lem}

Now, we derive expressions for the decompositions of $JS_H$ and $P(\pi_K^H)$ induced by the group algebra decomposition of $JS$.
For $\alpha \in \mathbb{Z}^{|n}$, let
\begin{equation}
    H_\alpha = \braket{s, r^{n/\alpha}},\
    K_\alpha = \braket{sr, r^{n/\alpha}}\ \text{and}\
    C_\alpha = \braket{r^{n/\alpha}},
\end{equation}
be subgroups of $\mathbf{D}_n$.
In fact, they cover all of the subgroups of $\mathbf{D}_n$ modulo conjugation.
Note that $H_n = K_n = \mathbf{D}_n$.
If $\alpha$ is a proper divisor of $\beta \in \mathbb{Z}^{|n}$ then
\begin{equation}
    C_\alpha < H_\alpha < H_\beta,\
    C_\alpha < K_\alpha < K_\beta\ \text{and}\
    C_\alpha < C_\beta,
\end{equation}
with associated intermediate coverings
\begin{equation}
    S_{C_\alpha} \to S_{H_\alpha} \to S_{H_\beta},\
    S_{C_\alpha} \to S_{K_\alpha} \to S_{K_\beta}\ \text{and}\
    S_{C_\alpha} \to S_{C_\beta}.
\end{equation}
We observe that, modulo conjugation, all possible group inclusions and intermediate coverings are depicted above. 

\begin{pro}
    For $\alpha \in \mathbb{Z}^{|n}$ we write $Q_\alpha = \mathbb{Z}^{|n/\alpha} - \set{1,2}$.
    Then
    \begin{enumerate}
        \item
        $JS_{H_\alpha} \sim
        \begin{cases}
            JS_{\mathbf{D}_n} \times \prod_{q \in Q_\alpha} B(q) & \text{if}\ \frac{n}{\alpha}\ \text{is odd} \\
            JS_{\mathbf{D}_n} \times B_3 \times \prod_{q \in Q_\alpha} B(q) & \text{if}\ \frac{n}{\alpha}\ \text{is even}
        \end{cases}$
        \item 
        $JS_{K_\alpha} \sim
        \begin{cases}
            JS_{\mathbf{D}_n} \times \prod_{q \in Q_\alpha} B(q) & \text{if}\ \frac{n}{\alpha}\ \text{is odd} \\
            JS_{\mathbf{D}_n} \times B_4 \times \prod_{q \in Q_\alpha} B(q) & \text{if}\ \frac{n}{\alpha}\ \text{is even}
        \end{cases}$
        \item
        $JS_{C_\alpha} \sim
        \begin{cases}
            JS_{\mathbf{D}_n} \times B_2 \times \prod_{q \in Q_\alpha} B(q)^2 & \text{if}\ \frac{n}{\alpha}\ \text{is odd} \\
            JS_{\mathbf{D}_n} \times B_2 \times B_3 \times B_4 \times \prod_{q \in Q_\alpha} B(q)^2 & \text{if}\ \frac{n}{\alpha}\ \text{is even}
        \end{cases}$
    \end{enumerate}
    If $n$ is odd then we just discard the components $B_3$ and $B_4$ ($\tfrac{n}{\alpha}$ is odd).
\end{pro}

\begin{proof}
    Assume that $n$ is even.
    By \eqref{indu}, we have that
    \begin{equation}\textstyle
        JS_{H_\alpha} \sim JS_{\mathbf{D}_n} \times B_2^{u_2} \times B_3^{u_3} \times B_4^{u_4} \times \prod_{q \in \mathbb{Z}^{|n} - \set{1,2}} B(q)^{u(q)},
    \end{equation}
    where $u_j = d_{\psi_j}^{H_\alpha}$ for $j \in \set{2,3,4}$, and $u(q) = d_{\rho^{n/q}}^{H_\alpha}$ for $q \in \mathbb{Z}^{|n} - \set{1,2}$.
    By Lemma~\ref{lem-dihedral_fix_dim}, one has that $u_2 = 0$, $u_3 = \delta$, $u_4 = 0$ and $u(q) = \varepsilon(q)$, where
    \begin{equation}
        \delta =
        \begin{cases}
            1 & \text{if}\ \frac{n}{\alpha}\ \text{is even} \\
            0 & \text{if}\ \frac{n}{\alpha}\ \text{is odd}
        \end{cases}
        \quad \text{and} \quad
        \varepsilon(q) =
        \begin{cases}
            1 & \text{if}\ q \in \mathbb{Z}^{|n/\alpha} \\
            0 & \text{if}\ q \notin \mathbb{Z}^{|n/\alpha}
        \end{cases}
    \end{equation}
    and the isogeny (1) follows.
    The conclusion for the remaining cases, as well as for the case $n$ odd, are obtained analogously. 
\end{proof}

In a very similar way, one has the following result (whose proof we omit for the sake of conciseness of the exposition).

\begin{pro}\label{pro:dihedral_Prym_dec}
    For $\alpha, \beta \in \mathbb{Z}^{|n}$ we write $$Q_{\alpha,\beta} = (\mathbb{Z}^{|n/\alpha} - \mathbb{Z}^{|n/\beta}) - \set{2}.$$
    If $\alpha$ is a proper divisor of $\beta$, then
    \begin{enumerate}
        \item $P(\pi_{H_\beta}^{H_\alpha}) \sim
        \begin{cases}
            B_3 \times \prod_{q \in Q_{\alpha, \beta}} B(q) & \text{if}\ \frac{n}{\alpha} + \frac{n}{\beta}\ \text{is odd} \\
            \prod_{q \in Q_{\alpha, \beta}} B(q) & \text{if}\ \frac{n}{\alpha} + \frac{n}{\beta}\ \text{is even}        
        \end{cases}$
        \item $P(\pi_{K_\beta}^{K_\alpha}) \sim
        \begin{cases}
            B_4 \times \prod_{q \in Q_{\alpha, \beta}} B(q) & \text{if}\ \frac{n}{\alpha} + \frac{n}{\beta}\ \text{is odd} \\
            \prod_{q \in Q_{\alpha, \beta}} B(q) & \text{if}\ \frac{n}{\alpha} + \frac{n}{\beta}\ \text{is even}
        \end{cases}$
        \item $P(\pi_{C_\beta}^{C_\alpha}) \sim P(\pi_{H_\beta}^{H_\alpha}) \times P(\pi_{K_\beta}^{K_\alpha})$.
        \item $P(\pi_{H_\beta}^{C_\alpha}) \sim B_2 \times P(\pi_{H_\beta}^{H_\alpha}) \times P(\pi_{\mathbf{D}_n}^{K_\alpha})$.
        \item $P(\pi_{K_\beta}^{C_\alpha}) \sim B_2 \times P(\pi_{K_\beta}^{K_\alpha}) \times P(\pi_{\mathbf{D}_n}^{H_\alpha})$.
    \end{enumerate}
    If $n$ is odd then we just discard the components $B_3$ and $B_4$ ($\tfrac{n}{\alpha} + \tfrac{n}{\beta}$ is even).
\end{pro}

The following corollary, which was proved in \cite{@CaroccaEtAl02}, follows directly from the previous proposition.

\begin{cor}\label{cor:Prym_1deg_comp}
    Let $S$ be a compact Riemann surface of genus $g \geqslant 2$ with a $\mathbf{D}_n$-action.
    Then for each $n$ we have that $$P(\pi_{\mathbf{D}_n}^{\braket{r}}) \sim B_2.$$
    Moreover, if $n$ is even, then $$P(\pi_{\mathbf{D}_n}^{\braket{s,r^2}}) \sim B_3 \, \mbox{ and } \, P(\pi_{\mathbf{D}_n}^{\braket{sr,r^2}}) \sim B_4.$$
\end{cor}

\subsection*{Prym affordable Jacobians}
\label{sec:Prym_aff}

Let $S$ be a compact Riemann surface of genus $g \geqslant 2$ with a $G$-action.
The group algebra decomposition of $JS$ is a powerful tool to study the following two questions:

\begin{que}
    When does the Jacobian $JS$ decomposes as a product of Jacobians of quotients of $S$? 
\end{que}

\begin{que}
    When does the Jacobian $JS$ decomposes as a product of $JS_G$ and Prym varieties of intermediate coverings of $\pi: S \to S_G$?
\end{que}

Preceding the group algebra decomposition, Kani and Rosen in \cite{@KaniRosen89} provided a partial answer for the first question. Later,  their results were generalized in \cite{@Reyes-CaroccaRodriguez19}.
The second question has been recently considered by Moraga in \cite{@Moraga24} for actions of affine groups over finite fields.

\s

    Following \cite{@Moraga24}, the group algebra decomposition of $JS$ with respect to $G$ is called \emph{affordable by Prym varieties} if each group algebra component of $JS$ with respect to $G$ is isogenous to the Prym variety of an intermediate covering.

\s
Here we introduce a more restrictive definition.

\begin{defi}
    A group $G$ is \emph{Prym-affordable} if every group algebra decomposition of a Jacobian with respect to $G$ is affordable by Prym varieties.
\end{defi}

Partial results for the problem of determining when $\mathbf{D}_n$ is Prym-affordable have been obtained.
Concretely, Carocca, Recillas  and Rodr\'iguez proved in \cite[Theorems 6.4 and 7.1]{@CaroccaEtAl02} that $\mathbf{D}_p$ ($p\geq3$ prime) and $\mathbf{D}_{2^e}$ ($e\geq2)$ are Prym-affordable.
By contrast, Lange and Recillas in \cite[§4.4]{@LangeRecillas04a} pointed out that $\mathbf{D}_{2p}$ is not Prym-affordable.

\s

The following result gives a complete answer to this  problem.

\begin{thm}\label{thm:dihedral_Prym_affordable}
    The dihedral group $\mathbf{D}_n$ is Prym-affordable if and only if $n=p^e$ for $p$ prime and $e \geqslant 1$.
\end{thm}

\begin{proof}
    Let $\alpha, \beta \in \mathbb{Z}^{|n}$. Consider the set  \begin{equation}
        Q_{\alpha, \beta} = (\mathbb{Z}^{|n/\alpha} - \mathbb{Z}^{|n/\beta}) - \set{2},
    \end{equation}and the subgroup $H_\alpha = \braket{s, r^{n/\alpha}}$ of $\mathbf{D}_n$. Let $n = p^e$ for $p$ prime and $e \geqslant 1$, and assume that $\mathbf{D}_n$ acts on a compact Riemann surface $S$ of genus $g \geqslant 2$.
    After considering Corollary~\ref{cor:Prym_1deg_comp}, it suffices to prove that for each $q \in \mathbb{Z}^{|n} - \set{1,2}$, the subvariety $B(q)$ of $JS$ is isogenous to the Prym variety of an intermediate covering.
    Set $\alpha_j = n/p^j$ and $\beta_j = n/p^{j-1}$ for $j \in \set{1, \ldots, e}$.
    Observe that
    \begin{equation}
        Q_{\alpha_j, \beta_j} = \mathbb{Z}^{|p^j} - \mathbb{Z}^{|p^{j-1}} = \set{p^j}.
    \end{equation}
    Then, by Proposition~\ref{pro:dihedral_Prym_dec}(1), one has that
    \begin{equation}\textstyle
        P(\pi_{H_{\beta_j}}^{H_{\alpha_j}}) \sim \prod_{q \in Q_{\alpha_j, \beta_j}} B(q) \sim B(p^j),
    \end{equation}
    for $1 \leqslant j \leqslant e$ if $p\geq3$, and for $2\leqslant j \leqslant e$ if $p=2$, as desired.
    
    \s
    
    Conversely, assume that $n$ is not the power of a prime.
    Then, one of the following cases occurs:
    \begin{multicols}{2}
    \begin{enumerate}[label=(\roman*)]
        \item $n \in pq \mathbb{Z}$ for $p,q \geqslant 3$ primes;
        \item $n \in 2p \mathbb{Z}$ for $p \geqslant 3$ prime.
    \end{enumerate}
    \end{multicols}

    We now proceed to prove that for each of the cases above, given a suitable choice of action, there is a group algebra component that is not isogenous to the Prym variety of an intermediate covering.
    Let $\theta: \Delta \to \mathbf{D}_n$ be a ske such that the quotient surface has genus $\gamma \geqslant 2$.    Then, by the Chevalley-Weil formula together with Proposition~\ref{pro-dihedral_dim_B(q)}, one has that all the group algebra components (associated to the action of $\theta$) have positive dimension.
    
\s

(i)  Assume that $n \in pq \mathbb{Z}$.
        By Proposition~\ref{pro:dihedral_Prym_dec}, if $B(pq) \sim P(\pi_K^H)$ for $H<K$ subgroups of $\mathbf{D}_n$ then, as every group algebra component has positive dimension, there are integers $\alpha, \beta \in \mathbb{Z}^{|n}$, with $\alpha$ a proper divisor of $\beta$, such that
        \begin{equation}
            Q_{\alpha, \beta} = (\mathbb{Z}^{|n/\alpha} - \mathbb{Z}^{|n/\beta}) - \set{2} = \set{pq}.
        \end{equation}
        It follows that $pq$ divides $n/\alpha$ but does not divide $n/\beta$.
        Since $p$ and $q$ are prime, at least one of them does not divide $n/\beta$ and $Q_{\alpha, \beta}$ is not a singleton; a contradiction.

\s

(ii)  Assume that $n \in 2p \mathbb{Z}$.
        By Proposition~\ref{pro:dihedral_Prym_dec}, if $B(pq) \sim P(\pi_K^H)$ for $H<K$ subgroups of $\mathbf{D}_n$ (and every group algebra component has positive dimension), then there are integers $\alpha, \beta \in \mathbb{Z}^{|n}$, with $\alpha$ a proper divisor of $\beta$, such that
        \begin{equation}
            Q_{\alpha, \beta} = (\mathbb{Z}^{|n/\alpha} - \mathbb{Z}^{|n/\beta}) - \set{2} = \set{2p}.
        \end{equation}
        In particular, $2p$ is a divisor of $\tfrac{n}{\alpha}$ (even) but not of $\tfrac{n}{\beta}$.
        Since $p \notin Q_{\alpha, \beta}$ it follows that $p$ divides $\tfrac{n}{\beta}$ and therefore $2$ does not divide $\tfrac{n}{\beta}$ (odd).
        We conclude that $\tfrac{n}{\alpha} + \tfrac{n}{\beta}$ is odd and thus
        \begin{equation}
            P(\pi_H^K) \sim B_3 \times B(2p)\ \text{or}\ 
            P(\pi_H^K) \sim B_4 \times B(2p);
        \end{equation}
        a contradiction.

\s
    This concludes the proof.
\end{proof}

\s

Moving forward, we refine our analysis to provide conditions under which a given  group algebra component is isogenous to a Prym variety. In order to state the result, we need to introduce some notation.

\begin{defi}
    For each subset $Q \subset \mathbb{Z}_+$, we define the function
    \begin{equation}
        L_Q: \mathbb{Z}_+ \to \mathbb{Z}_+\
        \text{given by}\
        L_Q(q) = \operatorname{lcm}( Q \cap \mathbb{Z}^{|q} - \set{q} ).
    \end{equation}
\end{defi}

It is easy to prove from the definition above, the following facts.
\begin{lem}\label{lem:L_Q}
    Set $Q \subset \mathbb{Z}_+$ and $q \in \mathbb{Z}_+$.

    \begin{enumerate}
        \item $L_Q(q)$ is a divisor of $q$.
        \item $L_Q(q) \neq q$ if and only if $L_Q(q)$ is a proper divisor of $q$.
        \item if $q$ is the power of a prime number then $L_Q(q) \neq q$.
        \item if $q \in Q$ then $L_Q(q) \neq q$ if and only if $Q \cap \mathbb{Z}^{|q} - \mathbb{Z}^{|L_Q(q)} = \set{q}$.
        \item if $q \in Q$ and $L_Q(q) \neq q$, then $L_Q(q) = \gcd \set{t \in \mathbb{Z}^{|q}: Q \cap \mathbb{Z}^{|q} - \mathbb{Z}^{|t} = \set{q}}$.
    \end{enumerate}
\end{lem}

\begin{defi}
    Let $\rho_a$ be the analytic representation of a dihedral action represented by a ske $\theta: \Delta \to \mathbf{D}_n$.
    We define
    \begin{equation}
        Q_\theta = \set{t \in \mathbb{Z}^{|n} - \set{1,2}: \braket{\rho_a, \rho^{n/t}} \geqslant 1}.
    \end{equation}
\end{defi}
Observe that  $q \in Q_{\theta}$ if and only if $\mbox{dim}B(q) \neq 0.$
\begin{thm}\label{thm:GAD_component_Prym}
    Let $n \geqslant 3$ be an odd integer, and let $S$ be a compact Riemann surface of genus $g \geqslant 2$ with a dihedral action represented by a surface-kernel epimorphism $\theta: \Delta \to \mathbf{D}_n$.
    For $q \in Q_\theta$, the subvariety $B(q)$ of $JS$ is isogenous to the Prym variety of an intermediate covering if and only if $L_{Q_{\theta}}(q) \neq q$.
    In this case,
    \begin{equation}
        B(q) \sim P(\pi_K^H)\ \text{for}\ H = \braket{s, r^q}\ \text{and}\ K = \braket{s, r^{L_{Q_\theta}(q)}}.
    \end{equation}
\end{thm}

\begin{proof}
    Assume that $L_{Q_\theta}(q) \neq q$.
    Set $\alpha = n/q$ and $\beta = n/L_{Q_\theta}(q)$, and recall that $H_\alpha = \braket{s, r^{n/\alpha}}$ and $Q_{\alpha, \beta} = \mathbb{Z}^{|n/\alpha} - \mathbb{Z}^{|n/\beta}$.
    By Lemma~\ref{lem:L_Q}(2), $\alpha$ is a proper divisor of $\beta$ and therefore $H_\alpha < H_\beta$.
    Moreover, by Lemma~\ref{lem:L_Q}(4),
    \begin{equation}
        Q_\theta \cap Q_{\alpha, \beta} = Q_\theta \cap ( \mathbb{Z}^{|q} - \mathbb{Z}^{|L_{Q_\theta}(q)}) = \set{q}.
    \end{equation}
    Then Proposition~\ref{pro:dihedral_Prym_dec}(1) implies that
    \begin{equation}\textstyle
        P \big( \pi_{H_\beta}^{H_\alpha} \big) \sim
        \prod_{t \in Q_{\alpha, \beta}} B(t) \sim
        \prod_{t \in Q_\theta \cap Q_{\alpha, \beta}} B(t) \sim
        B(q).
    \end{equation}

    For the converse, assume that $B(q) \sim P(\pi_H^K)$ for $H<K$ subgroups of $\mathbf{D}_n$.
    By Proposition~\ref{pro:dihedral_Prym_dec}, we can assume that $H = H_{n/q}$ and $K = H_{n/t}$ with $t$ a proper divisor of $q$ such that
    \begin{equation}
        Q_\theta \cap (\mathbb{Z}^{|q} - \mathbb{Z}^{|t}) = \set{q}.
    \end{equation}
    By Lemma~\ref{lem:L_Q}(5), we conclude that $L_{Q_\theta}(q)$ divides $t$ and $L_{Q_\theta}(q) \neq q$.    
\end{proof}

We end this subsection by remarking that there are group algebra decompositions of Jacobians with respect to non Prym-affordable dihedral groups that are affordable by Prym varieties. We construct explicit examples below.

\s 

Consider a dihedral action in genus $g \geqslant 2$ represented by a ske $\theta: \Delta \to \mathbf{D}_n$ with $n$ odd, and assume that the signature of the action is
 $(\gamma; 2^t, m_1, \ldots, m_v)$.
By Lemma~\ref{lem-rho1}, we see that $n \in Q_\theta$ and therefore  $Q_\theta \neq \emptyset$. Moreover,  it can be seen that  $$\gamma \geqslant 1, \mbox{ or } \gamma = 0 \mbox{ and }t \geqslant 6 \implies Q_\theta = \mathbb{Z}^{|n} - \set{1,2}.$$
In other words, in these cases all the subvarieties $B(q)$ of $JS$ have positive dimension (Proposition~\ref{pro-dihedral_dim_B(q)}).
Consequently, if we search for examples of group algebra decompositions with respect  to non Prym-affordable dihedral groups that are affordable by Prym varieties, we need to consider the following cases:
\begin{align}
    (1)\ \gamma = 0\ \text{and}\ t=2, &&
    (2)\ \gamma = 0\ \text{and}\ t=4.
\end{align}
(Because we should have  dimension zero subvarieties.)

\s

{\it Example.} Let $p,q \geqslant 3$ be distinct prime numbers.
    The (non Prym-affordable) group $\mathbf{D}_{p^2q}$ acts in genus $(p^2-1)(q-1)$ with signature $(0; 2, 2, q, p^2)$.
    By Theorem~\ref{thm-dihedral_odd_formula}, the analytic representation of the action is $$\rho_a \cong W(p^2q) \oplus W(pq).$$
    Note that $Q_\theta = \set{pq, p^2q},$  $L_{Q_\theta}(pq) = 1$ and $L_{Q_\theta}(p^2q) = pq$.
    By Theorem~\ref{thm:GAD_component_Prym}, we conclude that
    \begin{equation}
        JS \sim
        B(pq)^2 \times B(p^2q)^2 \sim
        P \big( \pi_{D_{p^2q}}^{H_p} \big)^2 \times P \big( \pi_{H_p}^{H_1} \big)^2
    \end{equation}which is a Prym affordable group algebra decomposition.

\subsection*{Completely decomposable Jacobians}

We recall that an abelian variety is \emph{completely decomposable} if it is isogenous to a product of elliptic curves.
Ekedahl and Serre in \cite{@EkedahlSerre93} posed the following two questions:

\setcounter{que}{0}
\begin{que}
    Is it true that, for all positive integers $g$, there exists a compact Riemann surface of genus $g$ whose Jacobian is completely decomposable?
\end{que}

\begin{que}
    Is the set of genera for which a compact Riemann surface with completely decomposable Jacobian exists infinite?
\end{que}

Since the publication of \cite{@EkedahlSerre93}, the study of completely decomposable Jacobians has attracted considerable interest not only in complex and algebraic geometry, but also in number theory. 
See, for example,  \cite{@CaroccaEtAl14, @Earle06, @Kani94, @MagaardEtAl09, @Nakajima07}.
Despite recent advancements, the previous questions still remain open.

\s

Notably, Ekedahl and Serre provided a list of genera (up to 1297) for which there is an algebraic curve exhibiting completely decomposable Jacobian, but significant gaps remained.
To address this, Paulhus and Rojas in \cite{@PaulhusRojas17} used computational tools to systematically construct examples in new genera, relying on the group algebra decomposition of Jacobian varieties.
See also the recent article \cite{@RodriguezRojas24}.

\s

Here we determine all Riemann surfaces with dihedral action such that the group algebra decomposition yields a complete decomposition of the Jacobian.

\begin{thm}\label{thm:comp_JacDec}
    Let $n \geqslant 3$ be a positive integer, and let $S$ be a compact Riemann surface of genus $g \geqslant 2$ with a $\mathbf{D}_n$-action.
    The group algebra decomposition of $JS$ with respect to $\mathbf{D}_n$ provides a complete decomposition of $JS$ if and only if one of the following cases occurs:
    {\tiny 
    \begin{longtable}{|c|c|l|l|r|}
        \hline
        $n$ & $g$ & signature & geometric signature & complete decomposition of $JS$ \\
        \hline
        \hline
        \endhead
        $3$ & $2$ & $(0; 2, 2, 3, 3)$ & & $B(3)^2$ \\
        & $3$ & $(0; 2, 2, 2, 2, 3)$ & & $B_2 \times B(3)^2$ \\
        & & $(1; 3)$ & & $JS_{\mathbf{D}_3} \times B(3)^2$ \\
        & $4$ & $(1; 2, 2)$ & & $JS_{\mathbf{D}_3} \times B_2 \times B(3)^2$ \\
        \hline
        $4$ & $2$ & $(0; 2, 2, 2, 4)$ & $(0; \braket{s}, \braket{sr}, \braket{r^2}, \braket{r})$ & $B(4)^2$ \\
        & $3$ & $(1; 2)$ & $(1; \braket{r^2})$ & $JS_{\mathbf{D}_4} \times B(4)^2$ \\
        & & $(0; 2, 2, 2, 2, 2)$ & $(0; \braket{s}^2, \braket{sr}^2, \braket{r^2})$ & $B_2 \times B(4)^2$ \\
        & & $(0; 2, 2, 4, 4)$ & $(0; \braket{sr}^2, \braket{r}^2)$ & $B_3 \times B(4)^2$ \\
        & & & $(0; \braket{s}^2, \braket{r}^2)$ & $B_4 \times B(4)^2$ \\
        & $4$ & $(0; 2, 2, 4)$ & $(0; \braket{s}, \braket{sr}^3, \braket{r})$ & $B_2 \times B_3 \times B(4)^2$ \\
        & & & $(0; \braket{s}^3, \braket{sr}, \braket{r})$ & $B_2 \times B_4 \times B(4)^2$ \\
        & $5$ & $(1; 2)$ & $(1; \braket{sr}^2)$ & $JS_{\mathbf{D}_4} \times B_2 \times B_3 \times B(4)^2$ \\
        & & & $(1; \braket{s}^2)$ & $JS_{\mathbf{D}_4} \times B_2 \times B_4 \times B(4)^2$ \\
        \hline
        $6$ & $2$ & $(0; 2, 2, 2, 3)$ & $(0; \braket{s}, \braket{sr}, \braket{r^3}, \braket{r^2})$ & $B(6)^2$ \\
        & $3$ & $(0; 2, 2, 2, 6)$ & $(0; \braket{sr}^2, \braket{r^3}, \braket{r})$ & $B_3 \times B(6)^2$ \\
        & & & $(0; \braket{s}^2, \braket{r^3}, \braket{r})$ & $B_4 \times B(6)^2$ \\
        & $4$ & $(0; 2^5)$ & $(0; \braket{s}, \braket{sr}^3, \braket{r^3})$ & $B_2 \times B_3 \times B(6)^2$ \\
        & & & $(0; \braket{s}^3, \braket{sr}, \braket{r^3})$ & $B_2 \times B_4 \times B(6)^2$ \\
        & & $(0; 2, 2, 3, 6)$ & $(0; \braket{s}, \braket{sr}, \braket{r^2}, \braket{r})$ & $B(3)^2 \times B(6)^2$ \\
        & $5$ & $(1; 3)$ & $(1; \braket{r^2})$ & $JS_{\mathbf{D}_6} \times B(3)^2 \times B(6)^2$ \\
        & & $(0; 2, 2, 2, 2, 3)$ & $(0; \braket{s}^2, \braket{sr}^2, \braket{r^2})$ & $B_2 \times B(3)^2 \times B(6)^2$ \\
        & & $(0; 2, 2, 6, 6)$ & $(0; \braket{sr}^2, \braket{r}^2)$ & $B_3 \times B(3)^2 \times B(6)^2$ \\
        & & & $(0; \braket{s}^2, \braket{r}^2)$ &  $B_4 \times B(3)^2 \times B(6)^2$ \\
        & $6$ & $(0; 2, 2, 2, 2, 6)$ & $(0; \braket{s}, \braket{sr}^3, \braket{r})$ & $B_2 \times B_3 \times B(3)^2 \times B(6)^2$ \\
        & & & $(0; \braket{s}^3, \braket{sr}, \braket{r})$ & $B_2 \times B_4 \times B(3)^2 \times B(6)^2$ \\
        & $7$ & $(1; 2, 2)$ & $(1; \braket{sr}^2)$ & $JS_{\mathbf{D}_6} \times B_2 \times B_3 \times B(3)^2 \times B(6)^2$ \\
        & & & $(1; \braket{s}^2)$ & $JS_{\mathbf{D}_6} \times B_2 \times B_4 \times B(3)^2 \times B(6)^2$ \\
        \hline
    \end{longtable}
    }
\end{thm}

\begin{proof}
   By Proposition~\ref{pro-dihedral_dim_B(q)}, if $\rho_a$ is the analytic representation of an action of $\mathbf{D}_n$ then $$\dim B(n) = \tfrac{1}{2} \phi(n) \braket{\rho_a, \rho^1}.$$
    The fact that $\phi(n) = 2$ if and only if $n \in \set{3,4,6}$ coupled with Lemma~\ref{lem-rho1} imply that $$n \neq 3,4,6 \, \implies \dim B(n) \geqslant 2.$$

    Assume that the group algebra decomposition with respect to $\mathbf{D}_3$ provides a complete decomposition of $JS$.
    Then, by \eqref{mesa}, the multiplicities of the irreducible $\mathbb{C}$-representations in $\rho_a$ are zero or one.
    Thus, $\rho_a$ is equivalent to either
    \begin{equation}
        \rho^1,\
        \psi_1 \oplus \rho^1,\
        \psi_2 \oplus \rho^1\ \text{or}\
        \psi_1 \oplus \psi_2 \oplus \rho^1.
    \end{equation}
    %
    Observe that
    \begin{equation}
    \widetilde{\Phi}_{\rho_a}(3) = \Phi_{\rho_a}(3) = \braket{\rho_a, \rho^1} - \braket{\rho_a, \psi_1 \oplus \psi_2} + 1.
    \end{equation}
    As a consequence of Theorem~\ref{thm-dihedral_rep_existence_odd}, each one of the representations above is the analytic representation of some $\mathbf{D}_3$-action.
    Obtaining the (geometric) signature of each such action is a straightforward application of Proposition~\ref{pro-dihedral_inverse_formula_odd}.

    \s
    
    For $\mathbf{D}_4$ and $\mathbf{D}_6$ the proof follows in the same way: employ Theorem~\ref{thm-dihedral_rep_existence_even} instead of Theorem~\ref{thm-dihedral_rep_existence_odd}, and apply Proposition~\ref{pro-dihedral_inverse_formula_even} to obtain the geometric signature.
\end{proof}

The following corollary is a quick consequence of the theorem above.

\begin{cor}
    If $n \geqslant 3$ is different from $3$, $4$ and $6$ then the group algebra decomposition with respect to each action of $\mathbf{D}_n$ does not provide a complete decomposition of the Jacobian.
\end{cor}

\subsection*{A generalization of the Ekedahl-Serre problem}

An interesting generalization of the Ekedahl-Serre problem involves seeking  Jacobians with an isogeny decomposition whose factors have the same dimension.

\begin{defi}
    Let $k$ be a nonnegative integer.
    An abelian variety is $k$-\emph{decomposable} if it is isogenous to a product of abelian varieties of dimension $k$.
    We say that such a product is a $k$-\emph{decomposition} of $JS$.
\end{defi}

\begin{lem}\label{lem:dim_B}
Let $\rho_a$ be the analytic representation of a $\mathbf{D}_n$-action in genus $g \geqslant 2$.
    For $q \in \mathbb{Z}^{|n} - \set{1,2}$ the following statements hold.
    \begin{enumerate}
        \item $\braket{\rho_a, \rho^1} \geqslant \braket{\rho_a, \rho^{n/q}}$.
        \item $\dim B(n) \geqslant 1$.
        \item $\dim B(n) \geqslant \dim B(q)$.
        \item If $\dim B(n) = \dim B(q)$ then $$q=n \mbox{ or } q=\tfrac{n}{2}  \mbox{ if }n \in 2\mathbb{Z} - 4\mathbb{Z}, \mbox{ and }
             \,\, q=n  \mbox{ otherwise }$$
\end{enumerate}
\end{lem}
\begin{proof} We recall that $$\dim B(q) = \tfrac{1}{2} \phi(q) \braket{\rho_a, \rho^{n/q}} \mbox{ and }\widehat{\Psi}_\theta(n) \geqslant \widehat{\Psi}_\theta(n) - \widehat{\Psi}_\theta(\tfrac{n}{q}).$$
   The proof of the first  statement follows after applying  Theorems~\ref{thm-dihedral_odd_formula} and \ref{thm-dihedral_even_formula}. The second statement follows from the first one and the fact that $\phi(n) \geqslant \phi(q)$. The third statement is an easy consequence of Lemma~\ref{lem-rho1}. Finally, by the first two statements we see that $$\dim B(n) = \dim B(q) \, \iff  \phi(n) = \phi(q) \mbox{ and } \braket{\rho_a, \rho^1} = \braket{\rho_a, \rho^{n/q}}.$$
    To conclude, note that $\phi(n) = \phi(q)$ implies that $n=q$ or $n=2q$ (for $q$ odd).
\end{proof}

\begin{lem}\label{pro:kdec_dimB}
    Let $S$ be a compact Riemann surface of genus $g \geqslant 2$ with a $\mathbf{D}_n$-action.
    If the group algebra decomposition of $JS$ with respect to $\mathbf{D}_n$ yields a $k$-decomposition of $JS$, then $k = \dim B(n)$ is a multiple of $\frac{1}{2} \phi(n)$.
\end{lem}

\begin{proof}
    If the group algebra decomposition yields a $k$-decomposition of $JS$, then $\dim B(n)$ equals zero or $k$.
    By Lemma~\ref{lem:dim_B} one has that $\dim B(n) \geqslant 1$ and hence $k = \dim B(n)$.
\end{proof}

\begin{thm}\label{thm:kdec_Jac}
    Let $n \geqslant 3$ and $k \geqslant 2$ be positive integers, and let $S$ be a compact Riemann surface of genus $g \geqslant 2$ with a $\mathbf{D}_n$-action.
    The group algebra decomposition with respect to $\mathbf{D}_n$ provides a $k$-decomposition of $JS$ if and only if one of the following cases occurs:
    {\tiny \begin{longtable}{|c|c|c|l|l|r|}
        \hline
        $n$ & $g$ & $g/k$ & $\text{signature}$ & $\text{geometric signature}$ & $k\text{-decomposition of}\ JS$ \\
        \hline
        \hline
        $3$ & $2m$ & $2$ & $(0; 2, 2, 3^{m+1})$ & & $B(3)^2$ \\
        & $3m$ & $3$ & $(0; 2^{2m+2}, 3)$ & & $B_2 \times B(3)^2$ \\
        \hline
        $4$ & $2m$ & $2$ & $(0; 2^{m+2}, 4)$ & $(0; \braket{s}, \braket{sr}, \braket{r^2}^m, \braket{r})$ & $B(4)^2$ \\
        & $4m$ & $4$ & $(0; 2^{2m+2}, 4)$ & $(0; \braket{s}, \braket{sr}^{2m+1}, \braket{r})$ & $B_2 \times B_3 \times B(4)^2$ \\
        & $4m$ & $4$ & & $(0; \braket{s}^{2m+1}, \braket{sr}, \braket{r})$ & $B_2 \times B_4 \times B(4)^2$ \\
        \hline
        $5$ & $4m$ & $2$ & $(0; 2, 2, 5^{m+1})$ &  & $B(5)^2$ \\
        & $6$ & $3$ & $(0; 2^6)$ & & $B_2 \times B(5)^2$ \\
        \hline
        6 & $4m$ & $4$ & $(0; 2, 2, 3^m, 6)$ & $(0; \braket{s}, \braket{sr}, \braket{r^2}^m, \braket{r})$ & $B(3)^2 \times B(6)^2$ \\
        & $6m$ & $6$ & $(0; 2^{2m+2}, 6)$ & $(0; \braket{s}, \braket{sr}^{2m+1}, \braket{r})$ & $B_2 \times B_3 \times B(3)^2 \times B(6)^2$ \\
        & $6m$ & $6$ & & $(0; \braket{s}^{2m+1}, \braket{sr}, \braket{r})$ & $B_2 \times B_4 \times B(3)^2 \times B(6)^2$ \\
        \hline
        $p$ & $(p-1) m$ & $2$ & $(0; 2, 2, p^{m+1})$ & & $B(p)^2$ \\
        \hline
        $p^e$ & $p^{e-1}(p-1) m$ & $2$ & $(0; 2, 2, p^m, (p^e))$ & & $B(p^e)^2$ \\
        \hline
        $pq$ & $(p-1)(q-1)$ & $2$ & $(0; 2, 2, p, q)$ & & $B(pq)^2$ \\
        \hline
        $2^e$ & $2^{e-1}m$ & $2$ & $(0; 2^{m+2}, (2^e))$ & $(0; \braket{s}, \braket{sr}, \braket{r^{2^{e-1}}}^m, \braket{r})$ & $B(2^e)^2$ \\
        \hline
        $2p$ & $p-1$ & $2$ & $(0; 2, 2, 2, p)$ & $(0; \braket{s}, \braket{sr}, \braket{r^p}, \braket{r^2})$ & $B(2p)^2$ \\
        & $2 (p-1) m$ & $4$ & $(0; 2, 2, p^m, 2p)$ & $(0; \braket{s}, \braket{sr}, \braket{r^2}^m, \braket{r})$ & $B(p)^2 \times B(2p)^2$ \\
        \hline
    \end{longtable}}
    Here $m$ is a positive integer;
    if $n \in \set{3,4,5}$ then $m \geqslant 2$.
    For the last five rows of the table above we make, in descending order, the following assumptions:
    \begin{enumerate*}[label=(\roman*)]
        \item $p \geqslant 7$ is prime;
        \item $p \geqslant 3$ is prime and $e \geqslant 2$;
        \item $p,q \geqslant 3$ are distinct primes;
        \item $e \geqslant 3$; and
        \item $p \geqslant 5$ is prime.
    \end{enumerate*}
    Also, for $n = p^e$, we introduce the notation $(p^e)$ to emphasize that the integer $p^e$ appears one time in the signature.
    Similarly for $n = 2^e$.
\end{thm}

\begin{proof} We start the proof by showing that
    if for some $k \in \mathbb{Z}_+$ the group algebra decomposition with respect to $\mathbf{D}_n$ provides a $k$-decomposition of $JS$, then $n$ is either a prime power or the product of two primes. Let $\rho_a$ be the analytic representation of such a $\mathbf{D}_n$-action.
    We write $m := \braket{\rho_a, \rho^1}$, which is a positive integer by Lemma~\ref{lem-rho1}.
    We proceed by contradiction.
    Assume that $n$ is not a prime power nor the product of two primes.
    There are two cases to consider.

    \s

{\bf (1)} Assume $n \notin 2 \mathbb{Z} - 4 \mathbb{Z}$.
        By Lemma~\ref{lem:dim_B} one has that $\braket{\rho_a, \rho^{n/q}} = 0$ for all $q \in \mathbb{Z}^{|n} - \set{1, 2, n}$.
        It follows that
        $$\Phi_{\rho_a}(q) =
            \begin{cases}
                m - \braket{\rho_a, \psi_1 \oplus \psi_2} + 1 & \text{if}\ (n,q) = n \\
                m - \braket{\rho_a, \psi_3 \oplus \psi_4} & \text{if}\ n\ \text{is even and}\ (n,q) = \tfrac{n}{2} \\
                m & 1 < (n,q) < \tfrac{n}{2} \\
                0 & (n,q) = 1
            \end{cases}$$

        \begin{enumerate}
        \item[(a)] If $n$ is odd then there are distinct odd primes $p_1,p_2 \in \mathbb{Z}^{|n}$ such that $n \neq p_1p_2$ and therefore $\widetilde{\Phi}_{\rho_a}(p_1p_2) = -m < 0$.
        \item[(b)]  If $n$ is even then there is an odd prime $p \in \mathbb{Z}^{|n}$ such that $n \neq 2p$ and therefore  $\widetilde{\Phi}_{\rho_a}(2p) \leqslant - m < 0$.
        \end{enumerate}In both cases, the second statement of  Theorem~\ref{thm-dihedral_rep_existence_odd} is not satisfied;  a contradiction.

        \s
        
        {\bf (2)} Assume $n \in 2\mathbb{Z} - 4\mathbb{Z}$.
        Note that $\tfrac{n}{2}$ is odd.
        By Lemma~\ref{lem:dim_B} one has that $\braket{\rho_a, \rho^{n/q}} = 0$ for all $q \in \mathbb{Z}^{|n} - \set{1,2,\tfrac{n}{2},n}$, and $\braket{\rho_a, \rho^2} = 0$  or $m$.
        It follows that
        \begin{equation}
            \Phi_{\rho_a}(q) =
            \begin{cases}
                m - \braket{\rho_a, \psi_1 \oplus \psi_2} + 1 & \text{if}\ (n,q) = n \\
                m - \braket{\rho_a, \psi_3 \oplus \psi_4} & \text{if}\ n\ \text{is even and}\ (n,q) = \tfrac{n}{2} \\
                m & 2 < (n,q) < \tfrac{n}{2} \\
                m - \braket{\rho_a, \rho^2} & (n,q) = 2 \\
                0 & (n,q) = 1
            \end{cases}
        \end{equation}

\begin{enumerate}
    \item[(a)] If there are distinct odd prime numbers $p_1,p_2 \in \mathbb{Z}^{|n}$, then $\widetilde{\Phi}_{\rho_a}(p_1p_2) < 0$ and we conclude as in (1).
    \item[(b)] If $n = 2p^e$ for some odd prime $p$ and $e \geqslant 3,$ then $\widetilde{\Phi}_{\rho_a}(2p^{e-1}) = - m < 0$ and we conclude as in (1).
    \item[(c)] If $n=2p^2$ then$$ \widetilde{\Phi}_{\rho_a}(2p^2) = \braket{\rho_a, \psi_3 \oplus \psi_4} - \braket{\rho_a, \psi_1 \oplus \psi_2} + 1 - m,$$ $$
            \widetilde{\Phi}_{\rho_a}(p^2) = -\braket{\rho_a, \psi_3 \oplus \psi_4}, \,\,\,\,
            \widetilde{\Phi}_{\rho_a}(2p) = \braket{\rho_a, \rho^2} - m,$$ $$ 
            \widetilde{\Phi}_{\rho_a}(p) = m, \,\,
            \widetilde{\Phi}_{\rho_a}(2) = m - \braket{\rho_a, \rho^2}.$$Note that the statement (2) of Theorem~\ref{thm-dihedral_rep_existence_even} tells us that $$\widetilde{\Phi}_{\rho_a}(q) \geqslant 0  \, \mbox{ for } \, q \in \set{2p^2, p^2, 2p, p, 2}.$$
       Consequently,   $\braket{\rho_a, \psi_j} = 0$ for $j=1,\ldots,4$, $\braket{\rho_a, \rho^2} = m=1$.
        However, $$\operatorname{lcm}(\operatorname{Supp} \widetilde{\Phi}_{\rho_a}) = p \neq 2p^2,$$  contradicting  the statement (4) of Theorem~\ref{thm-dihedral_rep_existence_even}.
\end{enumerate}
   \s 

    Now, we proceed to determine all the geometric signatures of $\mathbf{D}_n$, for $n \geqslant 3$ a prime power or a product of two primes, whose associated group algebra decomposition provides a $k$-decomposition, for some $k \geqslant 2$.
    Our approach will be to apply Theorems \ref{thm-dihedral_rep_existence_odd} and \ref{thm-dihedral_rep_existence_even} to find all the compatible analytic representations.
    Then, we obtain the geometric  signature as an  application of Propositions~\ref{pro-dihedral_inverse_formula_odd} and \ref{pro-dihedral_inverse_formula_even}.

\s

    Let $\rho_a$ be the analytic representation of a $\mathbf{D}_n$-action as above.
    Lemma~\ref{pro:kdec_dimB} states that $k = \frac{1}{2} \phi(n) m$ with $m = \braket{\rho_a, \rho^1}$ a positive integer.
    Proposition~\ref{pro-dihedral_dim_B(q)} together with Lemma~\ref{lem:dim_B} imply that $\braket{\rho_a, \rho^{n/q}} = 0$ for $q \in \mathbb{Z}^{|n} - \set{1,2,n}$;
    unless $n \in 2\mathbb{Z} - 4 \mathbb{Z}$, in which case we can also have $\braket{\rho_a, \rho^2} = m$.
    Moreover, by Theorem~\ref{thm-dihedral_rep_existence_odd}(3) and Theorem~\ref{thm-dihedral_rep_existence_even}(3), $\braket{\rho_a, (\rho^{n/q})^\sigma} = \braket{\rho_a, \rho^{(n,h)}}$ for each element $\sigma$ of the Galois group of $\rho^{n/q}$.
   Summarizing, we have that
    \begin{equation}
        \braket{\rho_a, \rho^h} =
        \begin{cases}
            m & \text{if}\ (n,h) = 1 \\
            0\ \text{or}\ m & \text{if}\ n \in 2 \mathbb{Z} - 4 \mathbb{Z}\ \text{and}\ (n,h) = \tfrac{n}{2} \\
            0 & \text{otherwise}
        \end{cases}
    \end{equation}
    In addition, by Proposition~\ref{pro-dihedral_dim_B(q)} one has that $\braket{\rho_a, \psi_j} = 0$ or $k$.

\s

    We provide a step-by-step explanation of the rest of the proof for just one concrete case. The remaining ones are obtained in an analogous way.

\s

    Assume that $n=4$.
    Then, $k=m \geqslant 2$ and
    \begin{align}
        \widetilde{\Phi}_{\rho_a}(4) &= \braket{\rho_a, \psi_3 \oplus \psi_4} - \braket{\rho_a, \psi_1 \oplus \psi_2} + 1, \\
        \widetilde{\Phi}_{\rho_a}(2) &= m - \braket{\rho_a, \psi_3 \oplus \psi_4}.
    \end{align}
    The second statement  of Theorem~\ref{thm-dihedral_rep_existence_even} requires that $\widetilde{\Phi}_{\rho_a}(4)$ and  $\widetilde{\Phi}_{\rho_a}(2)$ are nonnegative.
    In particular, $\braket{\rho_a, \psi_3 \oplus \psi_4} \leqslant m$.

\begin{enumerate}
    \item  If $\braket{\rho_a, \psi_3 \oplus \psi_4} = 0$ then $\braket{\rho_a, \psi_i} = 0$ for $i=1,2,3,4.$ 
    \item If $\braket{\rho_a, \psi_3 \oplus \psi_4} = m$ then,  by the first  statement  of Theorem~\ref{thm-dihedral_rep_existence_even}, which says that $\braket{\rho_a, \psi_2} + 1 \geqslant \braket{\rho_a, \psi_1} + m$, we have that $\braket{\rho_a, \psi_1} = 0$ and $\braket{\rho_a, \psi_2} = m$.
\end{enumerate}We then  conclude, by Theorem~\ref{thm-dihedral_rep_existence_even}, that the compatible analytic representations are $m W(4)$, $m \psi_2 \oplus m \psi_3 \oplus m W(4)$ and $m \psi_2 \oplus m \psi_4 \oplus m W(4)$.\end{proof}

The following is a direct consequence  of the previous theorem.

\begin{cor}
    For $k \geqslant 2$, if $n$ is neither a prime power nor the product of two primes, then the group algebra decomposition with respect to each action of $\mathbf{D}_n$ does not provide a $k$-decomposition of the Jacobian.
\end{cor}

We end this section by pointing out two facts.

\s 

\begin{enumerate}
    \item Theorem \ref{thm:kdec_Jac} implies that the group algebra decomposition with respect to $\mathbf{D}_n$ provides a $2$-decomposition if and only if  one of the following cases occurs:
    {\tiny \begin{longtable}{|c|c|l|l|r|}
        \hline
        $n$ & $g$ & $\text{signature}$ & $\text{geometric signature}$ & $2\text{-decomposition of}\ JS$ \\
        \hline
        \hline
        $3$ & $4$ & $(0; 2, 2, 3^3)$ & & $B(3)^2$ \\
        & $6$ & $(0; 2^6, 3)$ & & $B_2 \times B(3)^2$ \\
        \hline
        $4$ & $4$ & $(0; 2^4, 4)$ & $(0; \braket{s}, \braket{sr}, \braket{r^2}^2, \braket{r})$ & $B(4)^2$ \\
        & $8$ & $(0; 2^6, 4)$ & $(0; \braket{s}, \braket{sr}^5, \braket{r})$ & $B_2 \times B_3 \times B(4)^2$ \\
        & $8$ & & $(0; \braket{s}^5, \braket{sr}, \braket{r})$ & $B_2 \times B_4 \times B(4)^2$ \\
        \hline
        $5$ & $4$ & $(0; 2, 2, 5)$ &  & $B(5)^2$ \\
        & $6$ & $(0; 2^6)$ & & $B_2 \times B(5)^2$ \\
        \hline
        $6$ & $8$ & $(0; 2, 2, 3, 3, 6)$ & $(0; \braket{s}, \braket{sr}, \braket{r^2}^2, \braket{r})$ & $B(3)^2 \times B(6)^2$ \\
        & $12$ & $(0; 2^6, 6)$ & $(0; \braket{s}, \braket{sr}^5, \braket{r})$ & $B_2 \times B_3 \times B(3)^2 \times B(6)^2$ \\
        & $12$ & & $(0; \braket{s}^5, \braket{sr}, \braket{r})$ & $B_2 \times B_4 \times B(3)^2 \times B(6)^2$ \\
        \hline
        $8$ & $4$ & $(0; 2, 2, 2, 8)$ & $(0; \braket{s}, \braket{sr}, \braket{r^4}, \braket{r})$ & $B(8)^2$ \\
        \hline
        $10$ & $4$ & $(0; 2, 2, 2, 5)$ & $(0; \braket{s}, \braket{sr}, \braket{r^5}, \braket{r^2})$ & $B(10)^2$ \\
        & $8$ & $(0; 2, 2, 5, 10)$ & $(0; \braket{s}, \braket{sr}, \braket{r^2}, \braket{r})$ & $B(5)^2 \times B(10)^2$ \\
        \hline
    \end{longtable}}
    \newpage
    It would be interesting to study if these actions could provide new examples of completely decomposable Jacobians, by considering the results in \cite{@AuffarthEtAl17} and \cite{@RodriguezRojas24}.
    
    \item A natural extension of the results of this section is the following problem: determine all those $\mathbf{D}_n$-actions giving rise to a group algebra decomposition in terms of factors of dimension at most $k.$ Particularly interesting would be to explore the case $k=2.$
\end{enumerate}

\paragraph{\bf Statements and Declarations} On behalf of all authors, the corresponding author states that there is no conflict of interest. This manuscript has no associated data.

\printbibliography[title=References]

\end{document}